\documentclass[12pt,reqno]{amsart}
\usepackage{amsthm,amsfonts,amssymb,epsfig,graphics,amsmath,epstopdf,color}

\usepackage{tikz}
\usepackage{pgfplots}

\newcommand{\ds}{\displaystyle}

\newcommand{\RR}{{\mathbb R}}

\newcommand{\NN}{{\mathbb N}}
\newcommand{\ZZ}{{\mathbb Z}}
\newcommand{\QQ}{{\mathbb Q}}
\newcommand{\Z}{{\mathbb Z}}

\newcommand{\CC}{{\mathbb C}}

\newcommand\cD{{\mathcal D}}

\newcommand\cR{{\mathcal R}}

\def\ee{\mathrm{e}}     
\def\ii{\mathrm{i}}     
\def\veps{\varepsilon}
\def\sech{\mathop{\mathrm{sech}}}

\newcommand{\bspm}{\left[\begin{smallmatrix}}\newcommand{\espm}{\end{smallmatrix}\right]}
\newcommand{\bspmb}{\left(\begin{smallmatrix}}\newcommand{\espmb}{\end{smallmatrix}\right)}
\newcommand{\bpm}{\left[\begin{matrix}}\newcommand{\epm}{\end{matrix}\right]}

\newtheorem{theo}{Theorem}[section]

\newtheorem{lemma}[theo]{Lemma}

\newtheorem{ass}[theo]{Assumption}

\newtheorem{remark}[theo]{Remark}

\newcommand{\PROOF}{\textbf{Proof.} }
\renewenvironment{proof}{\noindent\PROOF}{\hfill $\Box$}

\def\epsilon{\varepsilon}

\def\beq{\begin{equation}}
	\def\eeq{\end{equation}}

\begin{document}

	\title[Traveling modulating pulse solutions]{Traveling modulating pulse solutions
		with small tails for a nonlinear wave equation in periodic media}

	\author[T. Dohnal]{Tom\'{a}\v{s} Dohnal}
	\address[T. Dohnal]{Institut f\"ur Mathematik, Martin Luther University Halle-Wittenberg, Germany}
	\email{tomas.dohnal@mathematik.uni-halle.de}

	\author[D. E. Pelinovsky]{Dmitry E. Pelinovsky}
	\address[D. E. Pelinovsky]{Department of Mathematics and Statistics, McMaster University,	Hamilton, Ontario, Canada, L8S 4K1}
	\email{dmpeli@math.mcmaster.ca}

	\author[G. Schneider]{Guido Schneider}
	\address[G. Schneider]{Institut f\"ur Analysis, Dynamik und Modellierung, Universit\"at Stuttgart, 70569 Stuttgart, Germany}
	\email{guido.schneider@mathematik.uni-stuttgart.de}

	\date{\today}

	\maketitle

	\begin{abstract}
		Traveling modulating pulse solutions consist of a small amplitude pulse-like envelope moving with a constant speed and modulating a harmonic carrier wave.
		Such solutions can be approximated by solitons of an effective nonlinear Schr\"{o}dinger equation arising as the envelope equation.
		We are interested in a rigorous  existence proof of such solutions for a nonlinear wave equation with spatially periodic coefficients.
		Such solutions are quasi-periodic in a reference frame co-moving with the envelope. We use spatial dynamics, invariant manifolds, and
		near-identity transformations to construct such solutions on large domains in time and space.  Although the spectrum
		of the linearized equations in the spatial dynamics formulation contains infinitely many eigenvalues on the imaginary axis or in the worst case the complete imaginary axis,
		a small denominator problem is avoided when the solutions
		are localized on a finite spatial domain with small tails in far fields.
	\end{abstract}


	\section{Introduction}\label{intro}

	We consider the
	semi-linear wave equation
	\begin{equation} \label{model}
		\partial^2_t u(x,t) - \partial^2_x u(x,t) + \rho(x) u(x,t) = \gamma r(x) u(x,t)^3, \quad x,t\in \RR,
	\end{equation}
	where
	$ x,t,u(x,t) \in \mathbb{R} $, $ \gamma =
	\pm 1 $, and $\rho,r$ are bounded, $2\pi$-periodic, strictly positive, and even functions in the set
	\begin{equation}
		\label{functions}
		\mathcal{X}_0 = \{ \rho \in L^{\infty}(\mathbb{R}) : \quad \rho(x) = \rho(x+2\pi), \quad \rho(-x) = \rho(x), \quad \rho(x) \geq \rho_0 > 0, \quad \forall x \in \mathbb{R} \}.
	\end{equation}
	The purpose of this paper is to prove the existence of
	traveling modulating pulse solutions.

	\begin{remark}
		The semi-linear wave equation \eqref{model} can be considered as a
		phenomenological
		model for the description
		of electromagnetic waves in  photonic crystal fibers.
		Such fibers show a much larger (structural) dispersion than
		homogeneous glass fibers. As a consequence they are much better able
		to support nonlinear localized structures such as pulses than their homogeneous counterpart.
		Most modern technologies for the transport of information through glass fibers
		use these pulses, cf. \cite{Ik20}. Sending a light pulse corresponds to sending the digital information ``one" over the zero background. Physically
		such a pulse consists of a localized envelope which modulates  an underlying electromagnetic carrier wave.
	\end{remark}

	The traveling modulating pulse solutions will be constructed as bifurcations from the trivial solution $ u = 0 $. Hence we first consider the linear wave equation
	$$
	\partial^2_t u(x,t) -
	\partial^2_x u(x,t)  + \rho(x) u(x,t) = 0, \quad x,t\in \RR.
	$$
	Since $ \rho \in \mathcal{X}_0 $, the linear wave equation can be solved by the family of Bloch modes
	\begin{equation*}
		u(x,t) = e^{\pm \ii\omega_n (l)t} e^{\ii l x} f_n (l, x), \quad n \in \mathbb{N}, \quad l \in \mathbb{B},
	\end{equation*}
	where $\mathbb{B} := \mathbb{R} \backslash \mathbb{Z}$ and  where the pair $(\omega_n(l),f_n(l,\cdot))$ satisfies the eigenvalue problem
	\begin{equation} \label{ulm1a}
		-(\partial_x+ \ii l)^2 f_n(l,x) +
		\rho(x) f_n(l,x) = \omega^2_n(l) f_n(l,x), \quad x\in \RR
	\end{equation}
	subject to the boundary  conditions
	$$
	f_n(l,x) = f_n(l, x + 2\pi)
	\quad  \textrm{and} \quad
	f_n(l,x) = f_n(l+1, x ) e^{\ii x} \quad \forall l \in \mathbb{R}, \; \forall x \in \mathbb{R}.
	$$  The eigenfunctions $f_n$ are $L^2(0,2\pi)$-normalized according to
	$$
	\int_0^{2\pi} |f_n(l,x)|^2 dx = 1, \quad \forall n \in \NN, \;\; \forall l \in \mathbb{B}.
	$$
	The curves of eigenvalues $l \mapsto \omega_n (l)$ are ordered such that
	$$
	0 < \omega_1(l) \leq \omega_2(l) \leq \dots \leq \omega_n(l) \leq
	\omega_{n+1}(l)  \leq \dots \quad \forall l \in \mathbb{B},
	$$
	where $ \omega_n(l) \to \infty $ for $ n \to \infty $, cf. \cite{OW11}.
	The positivity of $ \omega_1(l) $ follows from positivity of $\rho \in \mathcal{X}_0$. A prototypical pattern of the curves of eigenvalues is shown on Figure \ref{fig1}.

	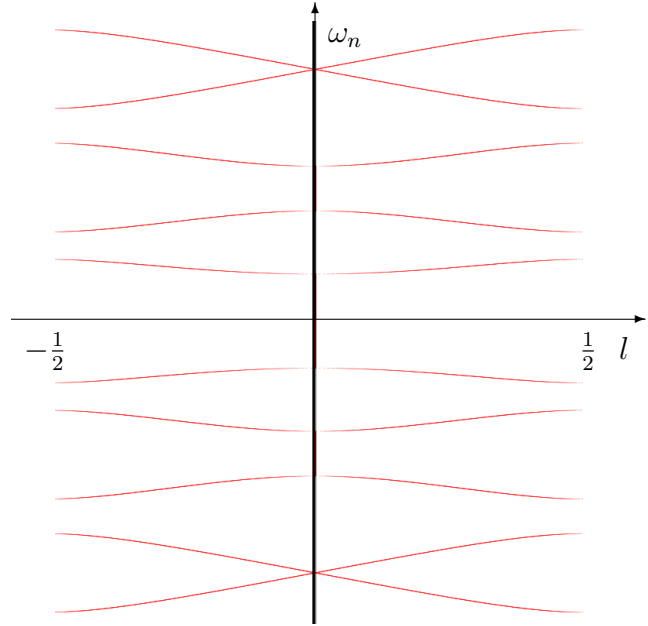
\begin{figure}[htbp]
		\begin{picture}(140,240)(0,-80)
			\put(-63,-80){\epsfig{file=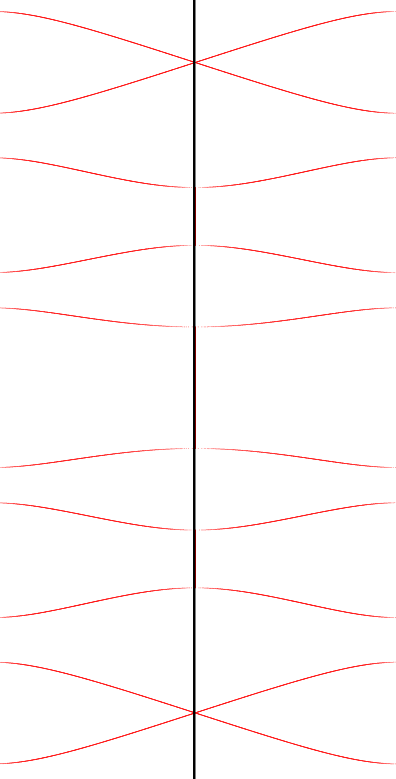,width=7cm,height=8cm}}
			\put(40,140){$ \omega_n $}
			\put(150,20){$ l $}
			\put(135,20){$ \frac{1}{2}$}
			\put(-75,20){$ -\frac{1}{2}$}
			\put(-80,35){\vector(1,0){240}}
			\put(35,35){\vector(0,1){120}}
		\end{picture}
		\caption{The curves of eigenvalues $\{ \pm \omega_n(l) \}_{n \in \mathbb{N}}$ plotted as functions of the Bloch wave numbers  $ l \in \mathbb{B}$ in a typical situation. }
		\label{fig1}
	\end{figure}

	The modulating pulse solutions  can be obtained via a weakly nonlinear  multiple scaling ansatz which results in the NLS equation for the description of slow temporal and spatial modulations of the envelope. In detail, for fixed $ n_0 \in \NN $ and $ l_0 \in \mathbb{B} $ solutions of the semi-linear wave equation \eqref{model} can be approximated by the	ansatz
	\begin{equation}
		\label{u-approx}
		u_{\rm app}(x,t) = \veps A (\veps (x-c_g t), \veps^2t) f_{n_0} (l_0, x) \ee^{\ii l_0
			x}  \ee^{-\ii
			\omega_{n_0}(l_0) t} + \text{\emph{c.c.}},
	\end{equation}
	with complex amplitude  $A = A(X,T)$, group velocity
	$ c_g := \omega'_{n_0}(l_0)$, and $0
	< \veps \ll 1$ being a small perturbation parameter. At the leading-order approximation, the envelope amplitude $A$
	satisfies the following NLS equation
	\begin{equation}
		\label{rr1}
		2 \ii \partial_T A + \omega''_{n_0} (l_0) \partial^2_X A +
		\gamma_{n_0}(l_0)  A|A|^2 = 0,
	\end{equation}
	where
	\[
	\gamma_{n_0} (l_0) = \frac{3 \gamma}{\omega_{n_0} (l_0)} \int^{2\pi}_0 r(x) |f_{n_0}
	(l_0,x)|^4 dx.
	\]
	The NLS equation (\ref{rr1}) possesses traveling pulse solutions if $\omega''_{n_0} (l_0) \gamma_{n_0}(l_0) > 0$ in the form:
	\begin{equation}
		\label{NLS-soliton}
		A(X,T) = \gamma_1 \sech (\gamma_2 (X-\tilde{c}T))
		\ee^{\frac{\ii  (2 \tilde{c} X - \tilde{c}^2 T)}{2\omega''_{n_0}(l_0)}}
		\ee^{-\ii \tilde{\omega} T}
	\end{equation}
	where $\tilde{c}$ and $\tilde \omega$ are arbitrary parameters such
	that $\tilde{\omega} \omega''_{n_0} (l_0) <0$ and the positive constants $ \gamma_1 $ and $ \gamma_2 $ are uniquely given by
	\begin{equation}
		\label{gamma-1-2}
		\gamma_1 = \sqrt{\frac{2 |\tilde{\omega}|}{|\gamma_{n_0}(l_0)|}}, \quad
		\gamma_2 = \sqrt{\frac{2 |\tilde{\omega}|}{|\omega_{n_0}'' (l_0)|}}.
	\end{equation}
	Without loss of generality, we can set $\tilde{c} = 0$ and $-\tilde{\omega}  = {\rm sgn}(\omega_{n_0}''(l_0)) = {\rm sgn}(\gamma_{n_0}(l_0))$, due to the scaling properties of the NLS equation (\ref{rr1}).

	\begin{remark}
		As an example consider the spatially homogeneous case with  $\rho(x) = 1$ and $r(x) = 1$,
		i.e.,  the semi-linear wave equation with constant coefficients. Then, we can re-order the eigenvalues
		and define
		\begin{equation}\label{E:om_fn}
			f_n(l,x) = \frac{1}{\sqrt{2\pi}} e^{\ii nx}, \quad
			\omega_n(l) := \sqrt{1+(n+l)^2}, \quad n \in \mathbb{Z}, \quad l \in \mathbb{B},
		\end{equation}
		producing
		\begin{equation}\label{E:c_om_gam}
			c_g = \omega_{n_0}'(l_0)= \frac{n_0+l_0}{\omega_{n_0}(l_0)}, \quad
			\omega_{n_0}''(l_0) = \frac{1}{\omega_{n_0}(l_0)^3}, \quad \gamma_{n_0}(l_0) = \frac{3 \gamma}{2 \pi \omega_{n_0}(l_0)}.
		\end{equation}
		The traveling pulse solutions exist for $\gamma = 1$ with $\tilde{\omega} = -1$ since $\omega_{n_0}''(l_0) > 0$.
	\end{remark}

	\begin{remark}
		In \cite{BSTU06} an approximation
		result was established that guarantees that wave-packet solutions of
		the semi-linear wave equation
		\eqref{model} with periodic coefficients can be approximated
		by solutions of the NLS equation (\ref{rr1}) on an $\mathcal{O}(\veps^{-2})$-time scale via $u_{\rm app}$ given by  (\ref{u-approx}). In \cite{DR20} this approximation was extended to the $d$-dimensional case.
	\end{remark}

	Existence of standing and moving modulating pulse solutions in homogenous and periodic media has been considered beyond the $\mathcal{O}(\veps^{-2})$-time scale. Depending on the problem, we have to distinguish between pulse solutions
	which decay to zero for $ |x| \to \infty  $ and generalized pulse solutions
	which have some small tails for large values of $|x|$.

	\begin{remark}
		In the spatially homogeneous case, i.e.
		if $ \rho = r = 1 $, the modulating pulse solutions are time-periodic
		in a frame co-moving with the envelope. Time-periodic solutions with finite energy are
		called breather solutions.
		However, it cannot
		be expected that such solutions with finite
		energy do exist in general,  according to the non-persistence of breathers result for
		nonlinear wave equations in homogeneous media \cite{De93,BKW94,Ma21}. Nevertheless, generalized breather solutions,
		i.e., modulating pulse solutions  with small tails, do exist.  Such solutions were
		constructed in \cite{GS01} with the help of spatial dynamics,
		invariant manifold theory and normal form theory.  In general, such solutions can only be constructed on
		large, but finite, intervals in $\RR$, cf. \cite{GS05,GS08}.
	\end{remark}

	\begin{remark}
		In the spatially periodic case standing generalized modulating pulse  solutions of the semi-linear wave equation (\ref{model}) have been constructed in \cite{LBCCS09}.
		These solutions are time-periodic, i.e., again breather solutions, but in contrast to the
		homogeneous case true spatially localized solutions can be constructed by properly tayloring the periodic coefficients.
		In \cite{BCLS11} breather solutions were constructed by
		spatial dynamics in the phase space of time-periodic solutions, invariant manifold
		theory and normal form theory. With the same approach in \cite{Ma20} such solutions were
		constructed for a cubic Klein-Gordon equation on an infinite periodic necklace graph.
		The existence of large amplitude breather solutions of the semi-linear wave equation \eqref{model} was shown in  \cite{HR19,MS21} via a variational approach. Breather solutions were recently considered in  \cite{Reichel22} for  quasi-linear wave equations with periodic coefficients.
	\end{remark}

	\begin{remark}
		To our knowledge traveling modulating pulse solutions have not been constructed before
		for the semi-linear wave equation \eqref{model} with spatially  periodic coefficients. For the  Gross-Pitaevski equation with a periodic potential
		such solutions were constructed in \cite{PS08} by using the coupled-mode approximation and in \cite[Chapter 5.6]{PelBook} by using the NLS approximation. The Gross-Pitaevski equation
		has a phase-rotational symmetry which is not present in the  semi-linear wave equation \eqref{model}. Another new aspect is the fact that in the present paper the normal form transformations are
		infinite-dimensional in contrast to the existing literature.
	\end{remark}

	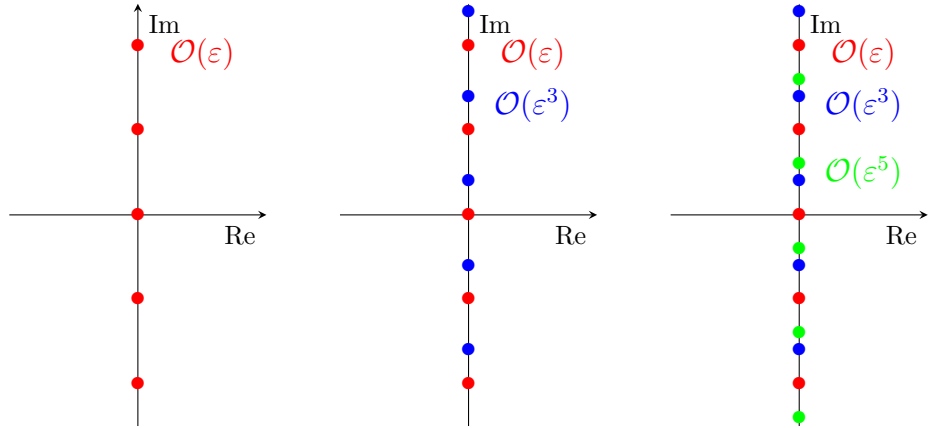
\begin{figure}[htbp] 
		\centering
				%
				%
		\begin{tikzpicture}
			\begin{axis}[width=5.0cm, height=7.2cm,
				xmin=-3.5, xmax=3.5,
				ymin=-5, ymax=5,
				axis lines=center,
				ticks=none,
				]
				\node[red] at (axis cs:0,0) {$\bullet$};
				\node[red] at (axis cs:0,2) {$\bullet$};
				\node[red] at (axis cs:0,4) {$\bullet$};
				\node[red] at (axis cs:0,-2) {$\bullet$};
				\node[red] at (axis cs:0,-4) {$\bullet$};
				\node[red] at (axis cs:1.75,3.8) {$\mathcal{O}(\varepsilon)$};
				\node[black, below left] at (axis cs:3.5,0) { \footnotesize Re };
				\node[black, below right] at (axis cs:0,5)  { \footnotesize Im };

			\end{axis}
		\end{tikzpicture}
		\qquad
		\begin{tikzpicture}
			\begin{axis}[width=5.0cm, height=7.2cm,
				xmin=-3.5, xmax=3.5,
				ymin=-5, ymax=5,
				axis lines=center,
				ticks=none,
				]
				\node[black, below left] at (axis cs:3.5,0) { \footnotesize Re };
				\node[black, below right] at (axis cs:0,5)  { \footnotesize Im };
				\node[red] at (axis cs:0,0) {$\bullet$};
				\node[red] at (axis cs:0,2) {$\bullet$};
				\node[red] at (axis cs:0,4) {$\bullet$};
				\node[red] at (axis cs:0,-2) {$\bullet$};
				\node[red] at (axis cs:0,-4) {$\bullet$};
				\node[red] at (axis cs:1.75,3.8) {$\mathcal{O}(\varepsilon)$};
				\node[blue] at (axis cs:0,0.8) {$\bullet$};
				\node[blue] at (axis cs:0,2.8) {$\bullet$};
				\node[blue] at (axis cs:0,4.8) {$\bullet$};
				\node[blue] at (axis cs:0,-1.2) {$\bullet$};
				\node[blue] at (axis cs:0,-3.2) {$\bullet$};
				\node[blue] at (axis cs:1.75,2.6) {$\mathcal{O}(\varepsilon^3)$};
			\end{axis}
		\end{tikzpicture}
		\qquad
		\begin{tikzpicture}
			\begin{axis}[width=5.0cm, height=7.2cm,
				xmin=-3.5, xmax=3.5,
				ymin=-5, ymax=5,
				axis lines=center,
				ticks=none,
				]
				\node[black, below left] at (axis cs:3.5,0) { \footnotesize Re };
				\node[black, below right] at (axis cs:0,5)  { \footnotesize Im };

				\node[red] at (axis cs:0,0) {$\bullet$};
				\node[red] at (axis cs:0,2) {$\bullet$};
				\node[red] at (axis cs:0,4) {$\bullet$};
				\node[red] at (axis cs:0,-2) {$\bullet$};
				\node[red] at (axis cs:0,-4) {$\bullet$};
				\node[red] at (axis cs:1.75,3.8) {$\mathcal{O}(\varepsilon)$};
				\node[blue] at (axis cs:0,0.8) {$\bullet$};
				\node[blue] at (axis cs:0,2.8) {$\bullet$};
				\node[blue] at (axis cs:0,4.8) {$\bullet$};
				\node[blue] at (axis cs:0,-1.2) {$\bullet$};
				\node[blue] at (axis cs:0,-3.2) {$\bullet$};
				\node[blue] at (axis cs:1.75,2.6) {$\mathcal{O}(\varepsilon^3)$};
				\node[green] at (axis cs:0,-0.8) {$\bullet$};
				\node[green] at (axis cs:0,-2.8) {$\bullet$};
				\node[green] at (axis cs:0,-4.8) {$\bullet$};
				\node[green] at (axis cs:0,1.2) {$\bullet$};
				\node[green] at (axis cs:0,3.2) {$\bullet$};
				\node[green] at (axis cs:1.75,1.0) {$\mathcal{O}(\varepsilon^5)$};
			\end{axis}
		\end{tikzpicture}
		\caption{Eigenvalues of the spatial dynamics formulation, see \eqref{spat-dyn} below, are dense on the imaginary axis. However,
			due to the convolution structure w.r.t. the $ z $-variable, see Theorem \ref{thm1},
			for  a certain power of $ \varepsilon $ only a part of the linear operator has to be taken into account.
			For controlling the order $ \mathcal{O}(\varepsilon) $ of the solution only the part $ A_1(\omega,c) $
			has to be considered. The central spectrum of  $ A_1(\omega,c) $ is sketched
			in the left panel. In the middle panel the central spectrum of $ A_1(\omega,c) $
			and $ A_3(\omega,c) $ is sketched. It plays a role  for controlling the order $ \mathcal{O}(\varepsilon^3) $.  The right panel shows a sketch of
			the central spectrum of $ A_1(\omega,c) $, $ A_3(\omega,c) $
			and $ A_5(\omega,c) $ which plays a role  for controlling the order $ \mathcal{O}(\varepsilon^5) $.
			In all cases there is a spectral gap between zero and the rest of the spectrum.
		}
		\label{newpic1}
	\end{figure}

	In the spatially periodic case traveling modulating solutions of the semi-linear wave equation (\ref{model}) in general are  quasi-periodic in the   frame co-moving with the envelope.
	Hence their  construction requires the use of three spatial variables
	rather than two spatial variables used in the previous works \cite{LBCCS09} and \cite{PS08}.  However,
	although the spectrum
	of the linearized equations in the spatial dynamics formulation contains infinitely many eigenvalues on the imaginary axis or in the worst case the complete imaginary axis, a small denominator problem
	is avoided by considering the problem on a finite spatial domain and by allowing for small tails, as illustrated in Figure \ref{newpic1}.

	The following result will be proven in this work. Figure \ref{fig5}
	illustrates the construction of a generalized modulating pulse solution as
	described in the following theorem.

	\begin{theo} \label{thm1}
		Let $ \rho, r \in \mathcal{X}_0 $ and $\gamma \neq 0$. Choose  $n_0 \in \NN $ and $ l_0 > 0 $ such that the following conditions
		are satisfied:
		\begin{equation}
			\label{non-degeneracy-1}
			\omega_{n}(l_0) \neq \omega_{n_0}(l_0), \qquad  \forall n \neq n_0,
		\end{equation}
		\begin{equation}
			\label{non-degeneracy-2}
			\omega_{n_0}'(l_0) \neq \pm 1, \qquad \omega_{n_0}''(l_0) \neq 0,
		\end{equation}
		and
		\begin{equation}
			\label{non-resonance}
			\omega_n^2(m l_0) \neq m^2 \omega_{n_0}^2(l_0), \quad m\in \{3,5, \dots 2N+1\}, \quad \forall n \in \mathbb{N},
		\end{equation}
		for some fixed $N \in \mathbb{N}$. If Assumption \ref{manifoldsass} below is satisfied, then there are $\veps_0 > 0$ and $C > 0$ such
		that for all $\veps \in (0, \veps_0)$  there exist traveling modulating  pulse solutions of the semi-linear wave equation \eqref{model} in the form
		\begin{equation} \label{GSansatz}
			u(x,t) = v(\xi,z,x) \quad \mbox{\rm with} \;\; \xi = x- c_gt, \;\;
			z = l_0 x - \omega t,
		\end{equation}
		where 	$c_g = \omega_{n_0}'(l_0)$, $\omega=\omega_{n_0}(l_0) + \widetilde{\omega} \veps^2$ with  $\widetilde{\omega}=-\text{\rm sgn}(\omega_{n_0}''(l_0)) = -\text{\rm sgn}(\gamma_{n_0}(l_0))$, and \\ $ v \in C^2([-\veps^{-(2N+1)},\veps^{-(2N+1)}],\mathcal{X})$  satisfies
		\begin{equation}
			\label{property-2}
			\ds \sup_{\xi \in [-\veps^{-(2N+1)},~\veps^{-(2N+1)}]} |v(\xi,z,x) - h(\xi,z,x)| \leq C \veps^{2N},
		\end{equation}
		where $\mathcal{X} := 
		H^2_{\rm per}(\mathbb{T},L^2(\mathbb{T})) \cap 
		H^1_{\rm per}(\mathbb{T},H^1_{\rm per}(\mathbb{T})) \cap 
		L^2(\mathbb{T},H^2_{\rm per}(\mathbb{T})) $ with $\mathbb{T} := \mathbb{R} \backslash \{2 \pi \mathbb{Z} \}$. \\The function $h \in C^2(\mathbb{R},\mathcal{X})$ satisfies 
		\begin{equation}
			\label{property-4}
	\lim_{|\xi| \to \infty} h (\xi,z,x) = 0  \quad {\rm and} \quad 
	\ds \sup_{\xi,z,x \in \RR} \left| h (\xi,z,x) - h_{\rm app}(\xi,z,x) \right|
			\leq C \veps^2,
		\end{equation}
		with
		\begin{equation} \label{qq1}
			h_{\rm app} (\xi,z,x) =  \veps \gamma_1 \sech (\veps
			\gamma_2 \xi) f_{n_0} (l_0,x) \ee^{\ii z} + \text{c.c.}.
		\end{equation}
		The constants $\gamma_1, \gamma_2$ are defined in \eqref{gamma-1-2},  and $ f_{n_0}(l_0,\cdot) \in H^2_{\rm per}(\mathbb{T})$ is a solution of 	\eqref{ulm1a}.
	\end{theo}
	\medskip

	\protect\begin{figure}[htb]

		\begin{picture}(140,160)(140,0)
			\put(35,0){\epsfig{file=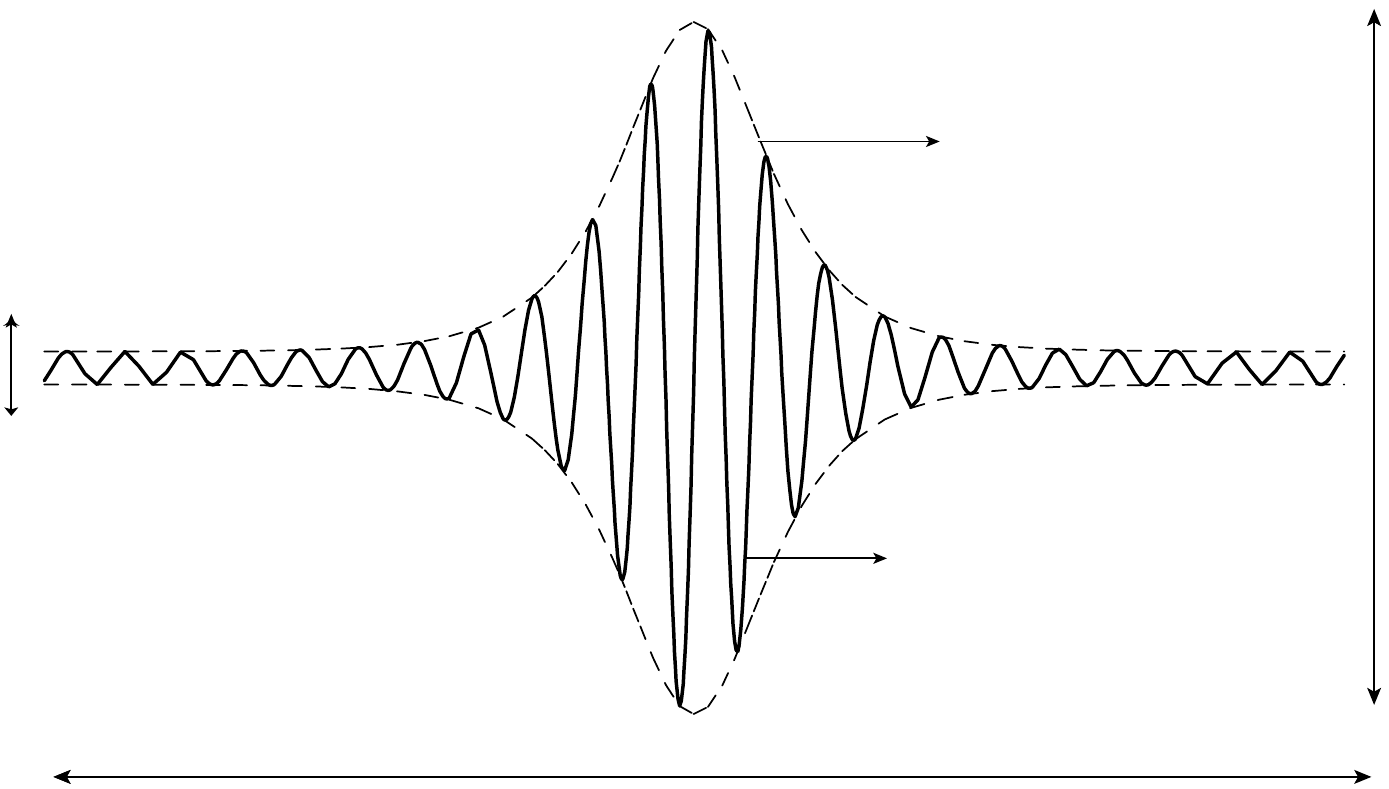,width=11cm,height=3cm} }

			\put(-10,43){$ \mathcal{O}(\varepsilon^{2N}) $}
			\put(355,43){$ \mathcal{O}(\varepsilon) $}
			\put(225,75){$ c_g $}
			\put(220,14){$ c_p$}
			\put(100,10){$\mathcal{O}(\veps^{-(2N+1)})$}
			\put(120,80){$\mathcal{O}(\veps^{-1})$}
			\put(180,95){\vector(1,0){90}}
			\put(200,95){\vector(-1,0){90}}
		\end{picture}
		\caption{A generalized modulating pulse solution as constructed in
			Theorem \ref{thm1} with $\mathcal{O} (\varepsilon^{2N})$ tails
			existing for $x $ in an interval of length
			$\mathcal{O}(\veps^{-(2N+1)})$ with an envelope
			advancing with group velocity $ c_g = \omega'_{n_0}(l_0) $,
			modulating a carrier wave advancing with phase velocity $ c_p  =  \omega_{n_0}(l_0) /l_0$,
			and leaves behind the standing periodic Bloch wave.
			The wavelength of the carrier wave and the period of the coefficients $\rho$, $r$ are of a comparable order.
		}
		\label{fig5}
		\protect\end{figure}

	\begin{remark}
		Assumption \ref{manifoldsass}  is of technical nature and guarantees the existence of infinite-dimensional invariant manifolds in the construction of the modulating pulse solutions. It is satisfied, for instance, if eigenvalues
		of the linearized operators on and near $\ii \mathbb{R}\setminus \{0\}$ are  semi-simple. An extended result can be obtained in the case of double eigenvalues, see Remark \ref{R:Jord-2} below.
	\end{remark}

	\begin{remark}
		The function $h$ solves a second-order differential equation that is
		an $\mathcal{O}(\veps)$-perturbation of the stationary NLS equation.
		We select $h$ to be a homoclinic orbit with  exponential decay to $0$ at infinity
		that is $\mathcal{O}(\veps^2)$-close to the NLS approximation (\ref{qq1}), cf. \eqref{property-4}, computed at the pulse solution (\ref{NLS-soliton}) for $\tilde{c} = 0$ and $\widetilde{\omega}=-\text{\rm sgn}(\omega_{n_0}''(l_0)) = -\text{\rm sgn}(\gamma_{n_0}(l_0))$.
	\end{remark}

	\begin{remark}
		If the non-resonance condition \eqref{non-resonance} is satisfied for all odd $m \geq 3$, then $N$ can be chosen arbitrarily large, but has to be fixed. The result of \cite{GS05} was
		improved in \cite{GS08} to exponentially small tails and
		exponentially long time intervals w.r.t.\ $\varepsilon$. It is not
		obvious that the exponential  smallness result can be transferred to the spatially periodic case.  We also do not use the Hamiltonian setup from
		\cite{GS01} because it is not clear how the Hamiltonian structure of
		the semi-linear wave equation \eqref{model} can be developed
		in the spatial dynamics formulation.
	\end{remark}

	The solution $v$ of Theorem \ref{thm1} is only defined on a large but finite spatial interval for the semi-linear wave equation (\ref{model}). However, due to the finite speed of propagation in the semi-linear wave equation (\ref{model}), it is also an approximate solution of the initial-value problem for a very long time, on a very large, but shrinking, spatial domain. The corresponding result is given by the following theorem.

	\begin{theo}
		\label{theorem-time}
		Let $v$ be the solution of Theorem \ref{thm1} and take an arbitrary function \\ $\phi\in C^2(\mathbb{R}\setminus[-\veps^{-(2N+1)},\veps^{-(2N+1)}],\mathcal{X})$ such that
		$$
		v_{\rm{ext}}(\xi,z,x) := \begin{cases}
			v(\xi,x,z), & (\xi,x,z) \in [-\veps^{-(2N+1)},\veps^{-(2N+1)}] \times \mathbb{R}\times \mathbb{R},\\
			\phi(\xi,x,z), & (\xi,x,z) \in (\mathbb{R}\setminus [-\veps^{-(2N+1)},\veps^{-(2N+1)}]) \times \mathbb{R}\times \mathbb{R},
		\end{cases}
		$$
		satisfies  $v_{\rm{ext}}\in C^2(\mathbb{R},\mathcal{X})$. Let
		$$
		u_0(x) := v_{\rm{ext}}(x,\ell_0 x,x) \quad \mbox{\rm and} \quad
		u_1(x) := -c_g \partial_{\xi} v_{\rm ext}(x,\ell_0 x,x) - \omega \partial_z v_{\rm ext}(x,\ell_0 x,x).
		$$
The corresponding solution of the semi-linear wave equation \eqref{model}
		with $u(\cdot,0) = u_0$ and $\partial_t u(\cdot,0) = u_1$ satisfies
		$$
		u(x,t) = v(x-c_g t, l_0 x - \omega t, x)
		$$
		for all $(x,t)\in  [-\veps^{-(2N+1)}, \veps^{-(2N+1)}] \times (0,\infty)$ such that $|x| + t <\veps^{-2N+1}$.
	\end{theo}

	\begin{remark}
		By Theorem \ref{theorem-time}, the modulated pulse solutions are approximated by $h_\text{app}$ much longer than on
		the $ \mathcal{O}(\veps^{-2}) $-time scale guaranteed by the approximation theorem given in \cite{BSTU06}. For instance, on the spatial interval $[-\tfrac{1}{2}\veps^{-(2N+1)}, \tfrac{1}{2}\veps^{-(2N+1)}]$ the approximation holds up to time $t=\tfrac{1}{2}\veps^{-(2N+1)}$.
	\end{remark}

	We shall describe the strategy of the proof of Theorems \ref{thm1} and \ref{theorem-time}. As  in \cite{GS01,GS05,GS08} the construction of the modulating pulse  solutions is based on a  combination of spatial dynamics, normal form transformations,  and invariant manifold theory. Plugging the ansatz \eqref{GSansatz} into \eqref{model}, we obtain an
	evolutionary system w.r.t. the unbounded space variable $\xi$, the spatial dynamics formulation, i.e.,
	we obtain a system of the form
	\begin{equation} \label{A10:SDF}
		\partial_{\xi} \widetilde{u}= M(\partial_z,\partial_x,x)\widetilde{u} + \widetilde{N}(\partial_z,\partial_x,x,\widetilde{u}),
	\end{equation}
	with $ M  \widetilde{u}$ linear and $ N $ nonlinear in $ \widetilde{u} $ which is a vector containing $ v $
	and derivatives of $ v $.
	For all values of the bifurcation parameter $ 0 < \varepsilon \ll 1 $ there are infinitely many eigenvalues of $M(\partial_z,\partial_x,x)$ on the imaginary axis, cf. Figure~\ref{newpic1}, and hence the center manifold reduction is of no use. However, the system is of the form
	\begin{align}
		\partial_{\xi} \widetilde{u}_0 & = M_0\widetilde{u}_0 + \widetilde{N}_0(\widetilde{u}_0,\widetilde{u}_r),  \label{cm1}\\
		\partial_{\xi} \widetilde{u}_r & = M_r\widetilde{u}_r + \widetilde{N}_r(\widetilde{u}_0,\widetilde{u}_r) + H_r(\widetilde{u}_0), \label{cm2}
	\end{align}
	where $ \widetilde{u}_0 $ is a vector in $\mathbb{C}^2$ corresponding to the eigenvalues of $ M $ which are close to zero and where $ \widetilde{u}_r $ corresponds to the infinite-dimensional remainder, i.e., to all the  eigenvalues of $M$ which are bounded away from zero for small $ |\varepsilon| $.
	For $ \varepsilon = 0 $ all eigenvalues of $ M_0 $ are zero.
	The nonlinearity in the $ \widetilde{u}_r $-equation is split into two parts such that $ \widetilde{N}_r(\widetilde{u}_0,0)=0 $.
	By finitely many normal form transformations in the $ \widetilde{u}_r $-equation
	we can achieve
	that the remainder term $H_r$ in \eqref{cm2} has the property
	$ H_{r}(\widetilde{u}_{0}) = \mathcal{O}(|\widetilde{u}_{0}|^{2N+2}) $
	where $ N $ is an arbitrary, but  fixed number, if certain non-resonance conditions are satisfied, cf. Remark \ref{rem37}.
	Concerning orders of $\varepsilon$, we have
	$ \widetilde{u}_{0} = \mathcal{O}(\varepsilon) $ and $ \widetilde{u}_{r} = \mathcal{O}(\varepsilon^{2N+2}) $. Hence, the finite-dimensional subspace
	$ \{\widetilde{u}_r = 0\} $ is  approximately invariant, and setting the highest order-in-$\varepsilon$ term $H_r(\tilde u_0)=\mathcal{O}(\varepsilon^{2N+2})$ to $0$,
	we obtain the reduced system
	$$
	\partial_{\xi} \widetilde{u}_{0} =  M_{0}\widetilde{u}_{0} + \widetilde{N}_{0}(\widetilde{u}_{0},0).
	$$
	For the reduced system, a homoclinic solution inside the subspace $\{\widetilde{u}_r=0\}$ can be found, which bifurcates with respect to $\varepsilon$ from the trivial solution. The  persistence of this solution for the system \eqref{cm1}-\eqref{cm2} cannot be expected, since the finite-dimensional subspace
	$ \{\widetilde{u}_r = 0\} $ is  not truly invariant for \eqref{cm1}-\eqref{cm2}, and therefore the necessary intersection of the stable and unstable manifolds is unlikely to happen in an infinite-dimensional phase space. However, the approximate homoclinic orbit can be used to prove that the center-stable manifold intersects the fixed space of reversibility transversally which in the end allows us to
	construct a modulating pulse solution with the properties stated in Theorem \ref{thm1}.

	The proof of Theorem \ref{theorem-time} is based on the energy method in the backward light cone associated to the point $(x_0,t_0)$. The point is arbitrarily chosen in the upper half-plane inside applicability of the solution $ v $ obtained in Theorem \ref{thm1}, that is, for $\xi \in [-\veps^{-(2N+1)}, \veps^{-(2N+1)}]$.

	\vspace{0.25cm}

	{\bf Organization of the paper.}  In Section \ref{sec-2} we introduce  the spatial dynamics formulation by using  Fourier series and Bloch modes. We develop near-identity transformations in Section \ref{sec-3} for reducing the size of the tails and increasing the size of the spatial domain. A local center-stable manifold in the spatial dynamics problem is constructed in Section \ref{sec-4}. The proof of Theorem \ref{thm1} is completed in Section \ref{sec-5} by establishing an intersection of the center-stable manifold with the fixed space of reversibility. Theorem \ref{theorem-time} is proven in Section \ref{sec-6}.

	\vspace{0.25cm}

	{\bf Acknowledgement.}  The work of Dmitry E. Pelinovsky is partially supported by AvHumboldt Foundation. The work of Guido Schneider is partially supported by the Deutsche Forschungsgemeinschaft DFG through the SFB 1173 ''Wave phenomena'' Project-ID 258734477.

	\section{Spatial dynamics formulation}
	\label{sec-2}

	Here we introduce  the spatial dynamics formulation by using  Fourier series and Bloch modes.
	We fix $l_0 \in \mathbb{B}$ and define
	\begin{equation}
		\label{change-variables}
		u(x,t) = v(\xi,z,x) \quad \mbox{\rm with} \;\;  \xi = x- ct,
		\;\;  z = l_0 x- \omega t,
	\end{equation}
	where $\omega$ and $c$ are to be determined and $v(\xi,\cdot,\cdot)$ satisfies
	$$
	v(\xi,z+2\pi,x) = v(\xi,z,x+2\pi) = v(\xi,z,x), \quad \forall (\xi,z,x) \in \mathbb{R}^3.
	$$
	Inserting (\ref{change-variables}) into the semi-linear wave equation \eqref{model} and using the chain rule, we obtain a new equation for $v$:
	\begin{align}
		\nonumber
		& \left[ (c^2-1)\partial_\xi^2
		+ 2 (c \omega - l_0) \partial_{\xi} \partial_z
		-2\partial_\xi\partial_x
		+ (\omega^2 - l_0^2) \partial_z^2
		- 2l_0 \partial_z \partial_x - \partial_x^2 \right] v(\xi,z,x) \\
		& \qquad + \rho(x) v(\xi,z,x) = \gamma r(x) v(\xi,z,x)^3, \quad \xi \in \RR, \quad x,z\in [0,2\pi)_{\rm per}.
		\label{eq3}
	\end{align}
	In order to consider this equation as an evolutionary system
	with respect to $\xi  \in \mathbb{R}$, we use  Fourier
	series in $z$
	\begin{equation}
		\label{Fourier}
		v(\xi,z,x) = \sum_{m \in \mathbb{Z}} \tilde{v}_{m}(\xi,x)
		e^{\ii m z}, \quad
		\tilde{v}_{m}(\xi,x) = \frac{1}{2\pi} \int_{0}^{2\pi}
		v(\xi,z,x)  e^{-\ii m z} dz.
	\end{equation}
	Equation \eqref{eq3} is converted
	through the Fourier expansion (\ref{Fourier}) into the spatial dynamics system
	for every $c \neq \pm 1$:
	\begin{equation}
		\label{spat-dyn}
		\partial_\xi
		\left( \begin{array}{c}
			\tilde{v}_m  \\
			\tilde{w}_m \end{array}
		\right) = A_m(\omega,c) \left( \begin{array}{c}
			\tilde{v}_m  \\
			\tilde{w}_m \end{array}
		\right) -\gamma (1-c^2)^{-1}  \left( \begin{array}{c}
			0  \\  r(x)
			(\tilde{v} \ast \tilde{v}\ast \tilde{v})_m \end{array}
		\right),
	\end{equation}
	for $ \xi \in \mathbb{R} $, $ m \in \mathbb{Z} $, $ x \in \mathbb{T} = \mathbb{R} \backslash \{ 2\pi \mathbb{Z} \} $,
	where $\tilde{w}_m := \partial_{\xi} \tilde{v}_m$, $A_m(\omega,c)$ is defined by
	$$
	A_m(\omega,c) = \left( \begin{array}{cc}
		0 & 1 \\
		(1-c^2)^{-1} [-(\partial_x + \ii m l_0)^2 + \rho(x) - m^2 \omega^2] &
		2 (1-c^2)^{-1} [\ii m c \omega - (\partial_x + \ii m l_0)] \end{array}
	\right)
	$$
	and	the double convolution sum is given by
	$$
	(\tilde{v} \ast \tilde{v}\ast \tilde{v})_m := \sum_{m_1,m_2 \in \Z} \tilde{v}_{m_1}  \tilde{v}_{m_2}  \tilde{v}_{m-m_1-m_2}.
	$$
	The dynamical system (\ref{spat-dyn}) can also be written in the scalar form as
	\begin{equation}
		[ (c^2-1) \partial_{\xi}^2
		+ 2 \ii m c \omega \partial_{\xi} -2 (\partial_x  + \ii m l_0) \partial_{\xi} - m^2 \omega^2 - (\partial_x + \ii m l_0)^2 + \rho(x) ]
		\tilde{v}_m = \gamma r(x) (\tilde{v} \ast \tilde{v}\ast \tilde{v})_m.
		\label{eq3_fourrier}
	\end{equation}

	\begin{remark}
		If $\rho \in L^{\infty}_\text{per}(\mathbb{T})$, then the domain $\tilde{D}$ and the range $\tilde{R}$ of the linear operator $A_m(\omega,c) : \tilde{D}  \subset \tilde{R} \to \tilde{R}$ are given by
		\begin{align}
			\label{domain-range}
			\tilde{D} = H^2_\text{per}(\mathbb{T})\times H^1_\text{per}(\mathbb{T}), \qquad
			\tilde{R} = H^1_\text{per}(\mathbb{T})\times L^2(\mathbb{T}).
		\end{align}
		Solutions of the dynamical system (\ref{spat-dyn}) are then sought such that at each $\xi \in \mathbb{R}$ they lie in
		the phase space
		\begin{align}
			\cD:=\{(\tilde{v}_m, \tilde{w}_m)_{m \in \ZZ}  & \in  (\ell^{2,2}(\mathbb{Z},L^2(\mathbb{T}))
			\cap \ell^{2,1}(\mathbb{Z},H^1_{\rm per}(\mathbb{T}))
			\cap \ell^{2,0}(\mathbb{Z},H^2_{\rm per}(\mathbb{T})) )
			\notag \\
			& \quad \times ( \ell^{2,1}(\mathbb{Z},L^2(\mathbb{T}))
			\cap \ell^{2,0}(\mathbb{Z},H^1_{\rm per}(\mathbb{T}))) \}, \label{phase-space}
		\end{align}
		with the range in
		\begin{align}
			\cR :=\{(\tilde{f}_m, \tilde{g}_m)_{m \in \ZZ} \in (\ell^{2,1}(\mathbb{Z},L^2(\mathbb{T})) \cap  \ell^{2,0}(\mathbb{Z},H^1_{\rm per}(\mathbb{T}))) \times \ell^{2,0}(\mathbb{Z},L^2(\mathbb{T}))\},
			\label{range}
		\end{align}
		where $\ell^{2,k}(\mathbb{Z},H^s)$, with $k \in \mathbb{N}$, is a  weighted $\ell^2$-space
		equipped with
		the norm
		$$ \|(\tilde{v}_m)_{m \in \ZZ} \|_{\ell^{2,k}(\mathbb{Z},H^s)} = \left(\sum_{m \in \ZZ}   \|\tilde{v}_m \|_{H^s}^2 (1 + m^2)^{k} \right)^{1/2}.
		$$
		The solution map for the initial-value problem associated to the dynamical system (\ref{spat-dyn}) is then defined as $[0,\xi_0] \ni \xi \mapsto  (\tilde{v}_m, \tilde{w}_m)_{m \in \ZZ} \in C^1([0,\xi_0],\mathcal{D})$. The phase space $\mathcal{D}$ in (\ref{phase-space}) is equivalent to the function space $\mathcal{X}$ in Theorem \ref{thm1} under the Fourier series (\ref{Fourier}).
	\end{remark}

	\begin{remark}
		Real solutions $v = v(\xi,z,x)$
		after the Fourier expansion (\ref{Fourier}) enjoy the symmetry:
		\begin{equation}
			\label{E:k-sym}
			\tilde{v}_{-m}(\xi,x) = \overline{\tilde{v}}_{m}(\xi,x), \quad \forall m \in \Z, \;\; \forall (\xi,x) \in \mathbb{R}^2.
		\end{equation}
		The cubic nonlinearity maps the space of Fourier series  where only the odd Fourier modes are non-zero to the same space. Hence, we can look for solutions of the spatial dynamics system (\ref{spat-dyn}) in the subspace
		$$
		\cD_{\rm{odd}} := \{ (\tilde{v}_m, \tilde{w}_m)_{m \in \ZZ} \in \cD: \quad
		\tilde{v}_{2m} = \tilde{w}_{2m} = 0,\quad    \tilde{v}_{-m} = \overline{\tilde{v}}_{m},\quad    \tilde{w}_{-m} = \overline{\tilde{w}}_{m}, \quad \forall m\in\ZZ\}.
		$$
		Hence the components $(\tilde{v}_m, \tilde{w}_m)$ for $-m \in \NN_{\rm odd}$ can be obtained from the components $(\tilde{v}_m, \tilde{w}_m)$ for $m \in \NN_{\rm odd}$ by using the symmetry (\ref{E:k-sym}).
	\end{remark}

	\subsection{Linearized Problem}
	\label{S:lin_prob}

	Truly localized modulating pulse solutions satisfy
	$$
	\lim_{\xi \to \pm \infty} v(\xi,z,x)  = 0,
	$$
	i.e., such solutions are homoclinic to the origin with respect to
	the evolutionary variable $\xi$. If these solutions exist,
	they lie in the intersection of the stable
	and unstable manifold of the origin. However, the
	modulating pulse solutions are not truly localized
	because of the existence of the infinite-dimensional center manifold
	for the spatial dynamics system (\ref{spat-dyn}).

	The following lemma characterizes zero eigenvalues $\lambda$ of the operators $A_m(\omega_0,c_g)$, where $\omega_0 =\omega_{n_0}(l_0)$ and $c_g =\omega'_{n_0}(l_0)$.

	\begin{lemma}
		\label{lem-linear}
		Fix $N \in \mathbb{N}$. Under the non-degeneracy and the non-resonance assumptions (\ref{non-degeneracy-1}), (\ref{non-degeneracy-2}), and (\ref{non-resonance}) the operator $A_m(\omega_0,c_g)$ with $m \in \{1,3,\cdots,2N+1\}$ has a zero eigenvalue if and only if $m=1$. The zero eigenvalue is algebraically double and geometrically simple.
	\end{lemma}

	\begin{proof}
		Let $m\in \NN_{\rm odd}$. The eigenvalue problem $A_m(\omega_0, c_g) \bspmb V\\W\espmb=\lambda \bspmb V\\W\espmb$ can be reformulated in the scalar form:
		\begin{equation}
			[ -(\partial_x + \ii m l_0 + \lambda )^2 + \rho(x)]
			V(x) = (m\omega_0 - \ii c_g \lambda)^2 V(x).
			\label{linear-eq}
		\end{equation}
		Eigenvalues $\lambda$ are obtained by setting $V(x)=f_n(ml_0 -\ii \lambda,x)$ and using the spectral problem (\ref{ulm1a}) for $l\in \CC$, where both $f_n(l,x)$ and $\omega_n(l)$ are analytically continued in $l \in \CC$. The eigenvalues are the roots of the nonlinear equations
		\begin{equation} \label{speclin0}
			\omega_n^2(m l_0 - \ii \lambda) = (m\omega_0 - \ii c_g \lambda)^2, \quad \quad n \in \NN.
		\end{equation}
		Zero eigenvalues $\lambda = 0$ exist if and only if
		there exist solutions of the nonlinear equations $\omega^2_n(m l_0) = m^2 \omega_0^2$. Since $\omega_0 = \omega_{n_0}(l_0)$,
		$\omega_n^2(m l_0) = m^2 \omega_0^2$ is satisfied for $m = 1$ and $n = n_0$. Due to the non-degeneracy assumption (\ref{non-degeneracy-1}),  $\omega^2_n(l_0) = \omega_0^2$ does not hold for any other $n$. This shows the geometric simplicity of $ \lambda = 0 $ for $m = 1$. It follows from (\ref{non-degeneracy-1}) and (\ref{non-resonance}) that no other solutions of $\omega^2_n(m l_0) = m^2 \omega_0^2$ exist for $m \in \{1,3,\cdots,2N+1\}$.

		It remains to prove that the zero eigenvalue for $m = 1$ is algebraically double. To do so, we again employ   the equivalence of the eigenvalue problem (\ref{ulm1a}) and
		\eqref{linear-eq} for $ \lambda = 0 $, $ l= l_0 $, $ n = n_0 $, and $ m = 1 $.
		For $n = n_0$, and $l = l_0$, this equation and its two derivatives with respect to $l$ generate the following relations:
		\begin{align}
			\label{bloch-1}
			& [-(\partial_x + \ii l_0)^2 + \rho(x) - \omega_0^2 ] f_{n_0}(l_0,x) = 0,  \\
			\label{bloch-2}
			& [-(\partial_x + \ii l_0)^2 + \rho(x) - \omega_0^2 ] \partial_l f_{n_0}(l_0,x) = 2 \omega_0 c_g f_{n_0}(l_0,x) + 2 \ii (\partial_x + \ii l_0) f_{n_0}(l_0,x), \\
			& [-(\partial_x + \ii l_0)^2 + \rho(x) - \omega_0^2 ] \partial_l^2 f_{n_0}(l_0,x) = 4 \omega_0 c_g \partial_l f_{n_0}(l_0,x) + 4 \ii (\partial_x + \ii l_0) \partial_l f_{n_0}(l_0,x)  \nonumber \\
			\label{bloch-3}
			& \qquad \qquad \qquad  \qquad \qquad \qquad  \qquad \qquad + 2 (\omega_0 \omega_{n_0}''(l_0) + c_g^2 - 1) f_{n_0}(l_0,x).
		\end{align}
		The non-degeneracy condition (\ref{non-degeneracy-2}) implies that $c_g^2 \neq 1$. Computing the Jordan chain for $A_1(\omega_0,c_g)$ at the zero eigenvalue with the help
		of (\ref{bloch-1}) and (\ref{bloch-2}) yields
		\begin{equation}
			\label{chain-1}
			A_1(\omega_0,c_g) F_0 = 0, \quad F_0(x):=\left( \begin{array}{c}
				f_{n_0}(l_0,x)  \\
				0 \end{array}
			\right),
		\end{equation}
		and
		\begin{equation}
			\label{chain-2}
			A_1(\omega_0,c_g) F_1 = F_0, \qquad F_1(x):=\left( \begin{array}{c}
				- \ii \partial_l f_{n_0}(l_0,x)  \\
				f_{n_0}(l_0,x) \end{array}
			\right).
		\end{equation}
		We use $f_{n_0}$  and $\partial_lf_{n_0}$ to denote $f_{n_0}(l_0,\cdot)$ and $\partial_lf_{n_0}(l_0,\cdot)$ respectively.	It follows from (\ref{bloch-2}) and (\ref{bloch-3}) that
		\begin{align}
			\label{constr-1}
			& \omega_0 c_g - l_0 + \langle f_{n_0},\ii f_{n_0}' \rangle = 0, \\
			\label{constr-2}
			& \omega_0 \omega_{n_0}''(l_0) + c_g^2 - 1 + 2 (\omega_0 c_g - l_0) \langle f_{n_0}, \partial_l f_{n_0} \rangle + 2 \langle f_{n_0},\ii \partial_l f_{n_0}' \rangle = 0,
		\end{align}
		where $f_{n_0}'$ denotes $\partial_x f_n(l_0,\cdot)$ and where we have used the normalization $\| f_{n_0}(l_0,\cdot) \|_{L^2(0,2\pi)} = 1$.

		\begin{remark}
			Let us define $$ \langle \langle f, g\rangle \rangle:=\langle f_1,g_1 \rangle + \langle f_2,g_2 \rangle, $$ where $\langle \phi,\psi \rangle = \int_0^{2\pi} \bar{\phi} \psi dx$ is the standard inner product in $L^2(0,2\pi)$. With some abuse of notation in the following we write $\langle f, g\rangle  $ for $ \langle \langle f, g\rangle \rangle $.
		\end{remark}

		Using complex conjugation, transposition, and integration by parts, the adjoint operator to $A_1(\omega,c)$ in $L^2(0,2\pi)$ is computed as follows:
		\begin{equation}
			\label{30a}
			A_1^*(\omega,c) = \left( \begin{array}{cc}
				0 & (1-c^2)^{-1} [-(\partial_x + \ii l_0)^2 + \rho(x) - \omega^2] \\
				1 & -2 (1-c^2)^{-1} [\ii c \omega - (\partial_x + \ii l_0)] \end{array}
			\right),
		\end{equation}
		for which we obtain
		\begin{equation}
			\label{chain-3}
			A_1^*(\omega_0,c_g) G_0=0, \quad G_0(x) :=\frac{1}{\omega_0 \omega''_{n_0}(l_0)} \left( \begin{array}{c}
				2 [\ii c_g \omega_0 - (\partial_x + \ii l_0)] f_{n_0}(l_0,x)  \\
				(1-c_g^2) f_{n_0}(l_0,x) \end{array}
			\right),
		\end{equation}
		where the normalization has been chosen such that $\langle G_0, F_1 \rangle=1$
		due to the relation (\ref{constr-2}). Note also that $\langle G_0, F_0 \rangle = 0$ due to the relation (\ref{constr-1}).

		For the generalized eigenvector of $A_1^*(\omega_0,c_g)$
		we have
		\begin{align}
			\label{chain-4}
			A_1^*&(\omega_0,c_g) G_1 = G_0,
		\end{align}
		with
		\begin{align*}
			G_1&:=\frac{1-c_g^2}{\omega_0 \omega''_{n_0}(l_0)}\left( \begin{array}{c} f_{n_0} + 2 \ii (1-c_g^2)^{-1} [\ii c_g \omega_0 - (\partial_x + \ii l_0)] \partial_l f_{n_0}  \\
				\ii \partial_l f_{n_0} \end{array}
			\right) + \nu G_0, \\
			&= \frac{1-c_g^2}{\omega_0 \omega''_{n_0}(l_0)}\left( \begin{array}{c} f_{n_0} + 2\ii (1-c_g^2)^{-1} [\ii c_g \omega_0 - (\partial_x + \ii l_0)] (\partial_l f_{n_0} -\ii \nu f_{n_0}) \\
				\ii(\partial_l f_{n_0} - \ii \nu f_{n_0})\end{array}
			\right),
		\end{align*}
		where $\nu$ is chosen so that $\langle G_1, F_1\rangle = 0$. A direct calculation produces
		$$\nu = \frac{2\ii}{\omega_0 \omega''_{n_0}(l_0)} \left((1-c_g^2)~\text{Re}~ \langle f_{n_0}, \partial_l f_{n_0}\rangle -(c_g\omega_0-l_0) \|\partial_l f_{n_0}\|^2 - \text{Im}\langle \partial_x\partial_l f_{n_0}, \partial_l f_{n_0}\rangle \right).$$

		As $A_1(\omega_0,c_0)$ has a compact resolvent, a standard argument using Fredholm's alternative guarantees that there exists a $2\pi$-periodic solution
		of the inhomogeneous equation
		$$
		A_1(\omega_0,c_g) \left( \begin{array}{l}
			\tilde{v}  \\
			\tilde{w} \end{array}
		\right) = F_1
		$$
		if and only if $F_1$ is orthogonal to ${\rm Ker}(A_1^*)$, i.e. to $G_0$. However, since $\langle G_0, F_1 \rangle = 1$, the Jordan chain for the zero eigenvalue terminates at the first generalized eigenvector (\ref{chain-2}), i.e.  $\lambda = 0$ is algebraically double.
	\end{proof}

	\begin{remark}
		By using the same argument, we verify that also the adjoint operator $A_1^*(\omega_0,c_g)$ has a double zero eigenvalue. This follows
		from the existence of solutions in (\ref{chain-3}) and (\ref{chain-4})
		and non-orthogonality of the generalized eigenvector $ G_1 $ of $A_1^*(\omega_0,c_g)$ to ${\rm ker}(A_1)$, i.e. to $F_0$ since
		$\langle G_1, F_0 \rangle = \langle G_1, A_1 F_1 \rangle = \langle A_1^* G_1, F_1 \rangle = \langle G_0, F_1 \rangle = 1$.
	\end{remark}

	In the low-contrast case, i.e., when the periodic coefficient $\rho$ is near $\rho \equiv 1$,
	the  non-degeneracy conditions (\ref{non-degeneracy-1}), (\ref{non-degeneracy-2}) and the non-resonance condition (\ref{non-resonance}) are easy to verify.
	If $\rho(x) = 1$, the eigenvalues $\omega_n(l)$ are known explicitly, see \eqref{E:om_fn}.
	The following lemma specifies the sufficient conditions under which
	the non-resonance assumption (\ref{non-resonance}) is satisfied.

	\begin{lemma}
		\label{lem-constant}
		Let $\rho(x) = 1 + \delta \rho_1(x)$ with $\rho_1(x) = \rho_1(x+2\pi)$ and $\delta$ being a constant parameter. There exists $\delta_0 > 0$ such that for every $\delta \in (-\delta_0,\delta_0)$ the non-degeneracy assumptions (\ref{non-degeneracy-1}) and (\ref{non-degeneracy-2}) are satisfied. The non-resonance assumption (\ref{non-resonance}) is satisfied if
		\begin{equation}
			\label{E:zero_ev_cond2}
			n_0 + l_0 \neq \frac{m^2 - 1 - \kappa^2}{2 m \kappa}, \quad \text{ where }\quad (m,\kappa) \in \{3,5,\dots,2N+1\}\times \Z . 
		\end{equation}
	\end{lemma}

	\begin{proof}
		The non-degeneracy assumptions (\ref{non-degeneracy-1}) and (\ref{non-degeneracy-2}) are satisfied because equation \eqref{E:c_om_gam}  for $\delta=0$  implies $c_g \in (-1,1)$  and  $\omega_{n_0}''(l_0) \neq 0$. As $\omega_{n_0}(l_0)$
		and $ \omega_{n_0}''(l_0)  $
		depend continuously on $\delta$, we get that (\ref{non-degeneracy-1}) and (\ref{non-degeneracy-2})   hold for $|\delta|$ small enough.

		For the non-resonance assumption (\ref{non-resonance}) we set $n=mn_0+\kappa$ with $m\in \{3, 5,\dots, 2N+1\}$,  $\kappa\in \ZZ$ and note that at $\delta=0$ we have
		\begin{align*}
			\omega_n^2(ml_0) &= 1+(mn_0+\kappa+ml_0)^2, \\ m^2\omega_{n_0}^2(l_0) &=m^2(1+(n_0+l_0)^2),
		\end{align*}
		see \eqref{E:om_fn}. Condition (\ref{non-resonance}) at $\delta=0$ is thus equivalent to
		\eqref{E:zero_ev_cond2}. As eigenvalues depend continuously on $\delta$, condition (\ref{non-resonance}) is satisfied for $|\delta|$ small enough if it is satisfied for $ \delta = 0 $.
	\end{proof}

	\begin{remark}
		The non-resonance condition \eqref{E:zero_ev_cond2} is satisfied for all $m\in \NN$
		if $ l_0\in \RR\setminus \QQ$.
	\end{remark}

	\subsection{Formal Reduction}

	Let us now consider a formal restriction of system \eqref{eq3_fourrier} to the subspace
	\begin{equation*}
		\mathcal{S} :=\{  (\tilde{v}_m, \tilde{w}_m)_{m \in \Z} \in \cD_{\rm{odd}}: \quad \tilde{v}_m = \tilde{w}_m = 0, \;\; m \in \ZZ_\text{odd}  \backslash \{-1,1\} \}
	\end{equation*}
	leading to the NLS approximation (\ref{qq1}). As $\mathcal{S}$ is not an invariant subspace of system \eqref{eq3_fourrier}, this reduction is only formal and a justification analysis has to be performed, which we do in the remainder of this paper.

	The nonlinear (double-convolution) term on $\mathcal{S}$ is given by
	$$
	(\tilde{v} \ast \tilde{v} \ast \tilde{v})_1 = 3 |\tilde{v}_1|^2 \tilde{v}_1.
	$$
	The scalar equation (\ref{eq3_fourrier}) on $\mathcal{S}$
	reduces to
	\begin{align*}
		& \left[ (c^2-1) \partial_{\xi}^2 + 2 \ii c \omega \partial_{\xi}
		- 2 \partial_{\xi} (\partial_x + \ii l_0) - \omega^2 - (\partial_x + \ii l_0)^2 + \rho(x) \right] \tilde{v}_1(\xi,x) \\
		& = 3 \gamma r(x) |\tilde{v}_1(\xi,x)|^2 \tilde{v}_1(\xi,x)
	\end{align*}
	for $ m = 1 $ and to the complex conjugate equation for $ m = -1 $.
	Using the Jordan block for the double zero eigenvalue in Lemma \ref{lem-linear},
	we write the two-mode decomposition:
	\begin{equation}
		\label{Bloch}
		\left\{
		\begin{array}{l} \tilde{v}_1(\xi,x) = \psi_1(\xi) f_{n_0}(l_0,x) - \ii
			\phi_1(\xi) \partial_l f_{n_0}(l_0,x), \\
			\tilde{w}_1(\xi,x) = \phi_1(\xi) f_{n_0}(l_0,x), \end{array} \right.
	\end{equation}
	where $\psi_1(\xi) = \veps A(X)$ with $X = \veps \xi$ and real $A(X)$.
	It follows from $\partial_{\xi} \tilde{v}_1(\xi,x) = \tilde{w}_1(\xi,x)$
	that $\phi_1(\xi) = \veps^2 A'(X)$, where the $\mathcal{O}(\veps^3)$ terms are neglected. Using $\omega = \omega_0 + \widetilde{\omega} \veps^2$ and $c = c_g$ with
	$\omega_0 = \omega_{n_0}(l_0)$ and $c_g = \omega_{n_0}'(l_0)$, we obtain the following equation at   order $\mathcal{O}(\veps^3)$:
	\begin{align}
		\label{E:NLS}
		(c_g^2-1) A'' f_{n_0} + 2 c_g\omega_0 A'' \partial_l f_{n_0}
		+ 2\ii A'' \partial_l f_{n_0}' - 2 \tilde{\omega} \omega_0 A f_{n_0} = 3 \gamma r  A^3 |f_{n_0}|^2 f_{n_0},
	\end{align}
	where equations (\ref{bloch-1}) and (\ref{bloch-2}) have been used and $f_{n_0}$ again denotes $f_{n_0}(l_0,\cdot)$. Projecting (\ref{E:NLS})  onto  $\text{span}\{f_{n_0}\}$ and using (\ref{constr-2}) yields the stationary NLS equation
	\begin{equation}
		\label{stat-NLS}
		- \omega_0 \omega_{n_0}''(l_0) A'' - 2  \omega_0 \widetilde{\omega} A =
		\omega_0 \gamma_{n_0}(l_0) A^3,
	\end{equation}
	which recovers the stationary version of the NLS equation (\ref{rr1})
	for $A(X,T)$ replaced by $A(X) e^{-\ii \widetilde{\omega}  T}$ with real $A(X)$. The modulated pulse solution
	corresponds to the soliton solution of the stationary NLS equation
	\eqref{stat-NLS},
	\beq
	\label{E:NLS-soliton}
	A(X) = \gamma_1 {\rm sech}(\gamma_2 X),
	\eeq
	where $\gamma_1$ and $\gamma_2$ are given by (\ref{gamma-1-2}). Note that among the  positive and decaying at infinity solutions of the stationary NLS equation (\ref{stat-NLS}) the pulse solution (\ref{E:NLS-soliton}) is unique up to a constant shift in $X$.

	\begin{remark}
		Unfolding the  transformations (\ref{change-variables}), (\ref{Fourier}), and (\ref{Bloch}) with $\psi_1(\xi) = \veps A(\veps \xi)$
		gives an approximation for $h_{\rm app}(\xi,z,x)$ on $\mathcal{S}$, see (\ref{qq1}).
	\end{remark}

	\subsection{Reversibility}
	\label{S:reversib}

	Because $\rho$ and $r$ are even functions in $\mathcal{X}_0$ given by  (\ref{functions}), the semi-linear wave equation (\ref{model}), being second order in space, is invariant under the parity transformation: $u(x,t) \mapsto u(-x,t)$. Similarly, since it is also second-order
	in time, it is invariant under the reversibility transformation:
	$u(x,t) \mapsto u(x,-t)$.

	The two symmetries are inherited by the scalar equations (\ref{eq3}) and (\ref{eq3_fourrier}): If $v(\xi,z,x)$ is a solution of (\ref{eq3}), so is
	$v(-\xi,-z,-x)$ and if $(\tilde{v}_m(\xi,x))_m$ is a solution of (\ref{eq3_fourrier}), so is $(\overline{\tilde{v}}_m(-\xi,-x))_m$. Since the symmetry is nonlocal in $x$, one can use the Fourier series in $x$ given by
	\begin{equation}
		\label{Fourier-2}
		\tilde{v}_m(\xi,x) = \sum_{k\in \Z} \hat{v}_{m,k}(\xi) e^{\ii kx}, \quad
		\hat{v}_{m,k}(\xi) = \frac{1}{2\pi} \int_{-\pi}^{\pi} \tilde{v}_m(\xi,x) e^{-\ii kx} dx,
	\end{equation}
	and similarly for  $ \tilde{w}_m  $ to rewrite the symmetry in the form:
	\begin{equation}
		\begin{array}{c}
			\mbox{\rm If } \{ \hat{v}_{m,k}(\xi),\hat{w}_{m,k}(\xi) \}_{(m,k) \in\NN_{\rm odd} \times \Z} \;\; \mbox{\rm is a solution of \eqref{spat-dyn} with (\ref{Fourier-2})}, \\ \mbox{\rm  so is }
			\{ \overline{\hat{v}}_{m,k}(-\xi), -\overline{\hat{w}}_{m,k}(-\xi) \}_{(m,k) \in\NN_{\rm odd} \times \Z}.
		\end{array}
		\label{reversibility-transform}
	\end{equation}
	The implication of the symmetry (\ref{E:k-sym}) and
	(\ref{reversibility-transform}) is that if a solution $\{ \hat{v}_{m,k}(\xi),\hat{w}_{m,k} (\xi) \}_{(m,k) \in\NN_{\rm odd} \times \Z}$
	constructed for $\xi \geq 0$ satisfies the reversibility constraint:
	\begin{equation}
		\label{reversibility-constraint}
		{\rm Im} \; \hat{v}_{m,k}(0) = 0, \quad {\rm Re} \;\hat{w}_{m,k}(0) = 0, \quad \forall (m,k) \in \NN_{\rm odd} \times \ZZ,
	\end{equation}
	then the solution $\{ \hat{v}_{m,k}(\xi), \hat{w}_{m,k}(\xi) \}_{(m,k) \in\NN_{\rm odd} \times \Z}$ can be uniquely continued for $\xi \leq 0$
	using the extension
	\begin{equation}
		\label{reversibility}
		\hat{v}_{m,k}(\xi) = \overline{\hat{v}}_{m,k}(-\xi), \quad
		\hat{w}_{m,k}(\xi) = -\overline{\hat{w}}_{m,k}(-\xi), \quad \forall \xi\in \RR_-.
	\end{equation}
	This yields a symmetric solution of the
	spatial dynamics system (\ref{spat-dyn})
	for every $\xi \in \RR$ after being reformulated with the
	Fourier expansion (\ref{Fourier-2}).

	\begin{remark}
		The pulse solution (\ref{E:NLS-soliton}) gives a leading order
		approximation (\ref{Bloch}) on $\mathcal{S} \subset  \cD_{\rm{odd}}$ which satisfies the reversibility constraint (\ref{reversibility-constraint}). Indeed, since $A'(0) = 0$,
		we have $\tilde{w}_1(0,x) = 0$ which implies $\hat{w}_{1,k}(0) = 0$.
		On the other hand, we have $\tilde{v}_1(0,x) = \veps A(0) f_{n_0}(l_0,x)$ with real $A(0)$ and generally complex $f_{n_0}(l_0,x)$. However, the Bloch mode $f_{n_0}(l_0,x)$ satisfies the symmetry
		$$
		\overline{f}_{n_0}(l_0,-x) = f_{n_0}(l_0,x),
		$$
		thanks to the non-degeneracy assumption (\ref{non-degeneracy-1}): If $f_{n_0}(l_0,x)$ is a solution of (\ref{ulm1a}) so is $\overline{f}_{n_0}(l_0,-x)$ and the eigenvalue $\omega_0^2 = \omega^2_{n_0}(l_0)$, is simple in the spectral problem (\ref{ulm1a}).
		Consequently, all Fourier coefficients of $f_{n_0}(l_0,x)$ are real
		which implies  ${\rm Im}~\hat{v}_{1,k}(0) = 0$.
		\label{remark-rev}
	\end{remark}

	\subsection{$ SO(2) $-symmetry}
	\label{secso2}

	The starting equation \eqref{eq3} for $ v = v(\xi,z,x) $ is translational invariant with respect to $ z \mapsto z + z_0 $, with $ z_0 \in \mathbb{R} $ arbitrary. This corresponds  to an invariance under the mapping $ (\tilde{v}_m,\tilde{w}_m) \mapsto (\tilde{v}_m,\tilde{w}_m)e^{imz_0} $ for all $ m \in \mathbb{Z}_{\textrm{odd}} $ in the Fourier representation 	\eqref{Fourier}. This symmetry allows us to restrict $A$
	for $\psi_1(\xi) = \varepsilon A(X)$ and $\phi_1(\xi) = \varepsilon^2 A'(X)$ in (\ref{Bloch}) to real-valued functions.

	\section{Near-identity transformations}
	\label{sec-3}

	By Lemma \ref{lem-linear}, the Fredholm operator $A_1(\omega_0,c_g)$
	has the double zero eigenvalue, whereas $A_m(\omega_0,c_g)$ for $3 \leq m \leq 2N+1$ admit no zero eigenvalues. In what follows, we decompose the solution
	in $X$ into a two-dimensional part corresponding to the double zero eigenvalue and the infinite-dimensional remainder term.

	\subsection{Separation of a two-dimensional problem}
	\label{S:sep_prob}

	Like in the proof of  Lemma \ref{lem-linear}, we denote the eigenvector and the generalized eigenvector of $A_1(\omega_0,c_g)$ for the double zero eigenvalue by $F_0$ and $F_1$, see (\ref{chain-1}) and (\ref{chain-2}),  and the eigenvector and the generalized eigenvector of $A_1^*(\omega_0,c_g)$ for the double zero eigenvalue by $G_0$ and $G_1$, see  (\ref{chain-3}) and (\ref{chain-4}).

	We define $\Pi$ as the orthogonal projection onto the orthogonal complement of the
	generalized eigenspace ${\rm span}(G_0,G_1)$, i.e.
	$$
	\begin{aligned}
		&\Pi : L^2(0,2\pi) \times L^2(0,2\pi) \to {\rm span}(G_0,G_1)^\perp, \\
		&\Pi \Psi := \Psi- \langle G_0, \Psi\rangle F_1
		-  \langle G_1, \Psi\rangle F_0.
	\end{aligned}
	$$
	The orthogonality follows from our normalization, which was chosen in the proof of Lemma \ref{lem-linear} so that  $\langle G_0, F_0 \rangle = \langle G_1, F_1 \rangle =0 $ and $ \langle G_0, F_1 \rangle = \langle G_1, F_0 \rangle = 1 $. Also note that ${\rm ker}~\Pi={\rm span} (F_0,F_1)$.
	Moreover, we have
	\begin{eqnarray*}
		A_1(\omega_0,c_g) \Pi \Psi & = & A_1(\omega_0,c_g) \Psi - \langle G_0, \Psi\rangle A_1(\omega_0,c_g) F_1
		-  \langle G_1, \Psi\rangle A_1(\omega_0,c_g) F_0 \\ & = &  A_1(\omega_0,c_g) \Psi - \langle G_0, \Psi\rangle  F_0 \\ & = & A_1(\omega_0,c_g) \Psi - \langle G_0,  A_1(\omega_0,c_g) \Psi\rangle  F_1
		-  \langle G_1, A_1(\omega_0,c_g) \Psi\rangle  F_0 \\ & = &  \Pi  A_1(\omega_0,c_g) \Psi.
	\end{eqnarray*}

	Compared to the two-mode decomposition (\ref{Bloch}), we write
	\begin{equation*}
		\left(
		\begin{array}{c} \tilde{v}_1(\xi,x) \\
			\tilde{w}_1(\xi,x) \end{array} \right) = \veps q_0(\xi)
		F_0(x)
		+ \veps q_1(\xi) F_1(x) + \veps S_1(\xi,x),
	\end{equation*}
	where $q_0,q_1:\RR\to\CC$ are unknown coefficients and where $S_1(\xi,\cdot) \in \tilde{D}$ for $\xi \in \mathbb{R}$ satisfies $\Pi S_1=S_1$, i.e.
	$$
	\langle G_0, S_1(\xi,\cdot) \rangle= \langle G_1, S_1(\xi,\cdot) \rangle = 0, \qquad \forall \xi \in \RR.
	$$
	Similarly, we write
	\begin{equation*}
		\left(
		\begin{array}{c} \tilde{v}_m(\xi,x) \\
			\tilde{w}_m(\xi,x) \end{array} \right) = \veps Y_m(\xi,x), \quad m \in \mathbb{N}_{\rm odd}\backslash \{1\}
	\end{equation*}
	and define $Y_1 := q_0 F_0 + q_1 F_1+ S_1$. Furthermore, we represent
	$Y_m = (V_m,W_m)^T$,
	i.e., $ \widetilde{v}_m = \varepsilon V_m $,  $ \widetilde{w}_m = \varepsilon W_m $,
	and use the notation ${\bf V} := (V_m)_{m \in \NN_{\rm odd}} $ and
	${\bf V}_{\geq 3}  := (V_m)_{m \in \NN_{\rm odd}\backslash \{1\}}$.

	For $ \omega = \omega_0 +\varepsilon^2 \widetilde{\omega} $ and $c = c_g$, we write
	$$
	A_m(\omega,c_g) = A_m(\omega_0,c_g) + \veps^2 \widetilde{\omega} (1-c_g^2)^{-1} B_m, \quad B_m =
	\left(
	\begin{array}{cc} 0 & 0 \\
		-m^2 (\omega + \omega_0) & 2\ii m c_g \end{array} \right).
	$$
	Because of
	$$\left(\begin{array}{cc} \langle G_0, F_0 \rangle & \langle G_0, F_1 \rangle \\ \langle G_1, F_0 \rangle & \langle G_1, F_1 \rangle \end{array} \right)=
	\left(\begin{array}{cc} 0 & 1\\ 1& 0\end{array} \right),$$
	the spatial dynamics system (\ref{spat-dyn}) with $\omega = \omega_0 + \widetilde{\omega} \veps^2$ and $c = c_g$ is now rewritten in the separated form:
	\begin{subequations}\label{E:sep_prob}
		\beq\label{E:sep_prob_q}
		\left( \begin{array}{c} \partial_\xi  q_1 \\ \partial_\xi  q_0 - q_1 \end{array} \right) = \veps^2  H_0(q_0,q_1,S_1,{\bf V_{\geq 3} }),
		\eeq
		\beq\label{E:sep_prob_Y1}
		\partial_\xi  S_1 =   \Pi   A_1(\omega_0,c_g) S_1
		+ \veps^2 \widetilde{\omega} (1-c_g^2)^{-1}  \Pi   B_1 Y_1
		+ \veps^2  \Pi   H_1(q_0,q_1,S_1,{\bf V}_{\geq 3} ),
		\eeq
		and for $m \in \NN_{\rm odd} \backslash \{1\}$,
		\beq\label{E:sep_prob_Yk}
		\partial_\xi   Y_m = A_m(\omega_0,c_g) Y_m
		+ \veps^2 \widetilde{\omega} (1-c_g^2)^{-1} B_m Y_m + \veps^2
		H_m(q_0,q_1,S_1,{\bf V}_{\geq 3} ),
		\eeq
	\end{subequations}
	where the correction terms  $H_0$ and $(H_m)_{m \in \NN_{\rm odd}}$ are given by
	\begin{align*}
		H_0 & = \frac{\widetilde{\omega}}{\omega_0\omega_{n_0}''(l_0)}
		\left( \begin{array}{c} -(\omega + \omega_0)q_0 + \ii
			((\omega + \omega_0) \langle f_{n_0}, \partial_l f_{n_0} \rangle + 2 c_g) q_1 \\
			\langle \partial_l f_{n_0}-\ii \nu f_{n_0}, f_{n_0} \rangle
			(\ii (\omega + \omega_0) q_0  +2c_gq_1)+ (\omega + \omega_0) \langle \partial_l f_{n_0}-\ii \nu f_{n_0}, \partial_l f_{n_0} \rangle q_1 \end{array} \right) \\
		& +\frac{1}{\omega_0\omega_{n_0}''(l_0)}  \left( \begin{array}{cc}
			\langle f_{n_0}, \widetilde{\omega} (B_1 S_1)_2 - \gamma r ({\bf V} \ast {\bf V} \ast {\bf V})_1 \rangle \\
			-\ii \langle\partial_l f_{n_0}-\ii \nu f_{n_0}, \widetilde{\omega} (B_1 S_1)_2 - \gamma r ({\bf V} \ast {\bf V} \ast {\bf V})_1 \rangle \end{array} \right),
	\end{align*}
	and
	\begin{align*}
		H_m & = -\gamma (1 - c_g^2)^{-1}  \left( \begin{array}{c} 0 \\
			r ({\bf V} \ast {\bf V} \ast {\bf V})_m \end{array} \right),
		\quad m \in \NN_{\rm odd}.
	\end{align*}

	\begin{remark}
		\label{rem-1}
		System (\ref{E:sep_prob}) does not have an invariant reduction at $S_1 = 0$ and  ${\bf V} _{\geq 3} = {\bf 0}$  because
		\begin{align*}
			& ({\bf V} \ast {\bf V} \ast {\bf V})_1 |_{S_1 = 0, {\bf V} _{\geq 3} = {\bf 0}} = 3 |q_0 f_{n_0} - \ii q_1 \partial_l f_{n_0} |^2 (q_0 f_{n_0} - \ii q_1 \partial_l f_{n_0}), \\
			& ({\bf V} \ast {\bf V} \ast {\bf V})_3 |_{S_1 = 0, {\bf V}_{\geq 3}  = {\bf 0}}  = (q_0 f_{n_0} - \ii q_1 \partial_l f_{n_0})^3,
		\end{align*}
		which contributes to $\Pi H_1$ and $H_3$ (as well as to $H_0$).
	\end{remark}

	\subsection{Resolvent operators for the linear system}

	In order to derive bounds (\ref{property-2}) and (\ref{property-4}), we need to perform near-identity transformations, which transform systems (\ref{E:sep_prob_Y1}) and (\ref{E:sep_prob_Yk}) to  equivalent versions but with residual terms of the order  $\mathcal{O}(\veps^{2(N+1)})$. To be able to do so, we will ensure
	that the operators $\Pi  A_1(\omega_0,c_g)  \Pi $ and $A_m(\omega_0,c_g)$, $3 \leq m \leq 2N+1$ are invertible with a bounded inverse.

	By Lemma \ref{lem-linear}, these operators do not have zero eigenvalues but  this is generally not sufficient since eigenvalues of these infinite-dimensional operators may accumulate near zero. However,
	the operators $A_m(\omega_0,c_g)$ have the special structure
	\begin{equation}
		\label{A-m}
		A_m(\omega_0,c_g) = \left( \begin{array}{cc} 0 & 1 \\
			L_m & M_m \end{array} \right), \quad \begin{array}{l}
			L_m = (1-c_g^2)^{-1} [-(\partial_x+ \ii m l_0)^2 + \rho(x) - m^2 \omega_0^2 ], \\
			M_m = 2 (1-c_g^2)^{-1} [\ii m c_g \omega_0 - (\partial_x + \ii m l_0)], \end{array}
	\end{equation}
	which we explore to prove invertibility of these operators  under the non-degeneracy and non-resonance conditions. The following lemma gives the result.

	\begin{lemma}
		\label{lem-invert}
		If the conditions (\ref{non-degeneracy-1}), (\ref{non-degeneracy-2}), and (\ref{non-resonance}) are satisfied, then there exists a $C_0 > 0$ such that
		\begin{equation}
			\label{bound-on-resolvent}
			\| ( \Pi  A_1(\omega_0,c_g)  \Pi)^{-1} \|_{\tilde{R} \to \tilde{D}} + \sum_{m=3}^{2N+1} \| A_m(\omega_0,c_g)^{-1} \|_{\tilde{R} \to \tilde{D}} \leq C_0,
		\end{equation}
		where $\tilde{D}$ and $\tilde{R}$ are defined in (\ref{domain-range}).
	\end{lemma}

	\begin{proof}
		Under the non-degeneracy condition (\ref{non-degeneracy-2}) which yields $c_g \neq \pm 1$, the entries of $A_m(\omega,c_g)$ are all non-singular. To ensure the invertibility of $A_m(\omega_0,c_g)$ with $m \in \mathbb{N}_{\rm odd} \backslash\{1\}$, we consider the resolvent equation
		$$
		A_m(\omega_0,c_g) \left( \begin{array}{l} v\\ w \end{array} \right) =
		\left( \begin{array}{l} f \\ g \end{array} \right),
		$$
		for a given $(f,g) \in \tilde{R}$. The solution $(v,w) \in \tilde{D}$ is given by $w = f$ and $v$ obtained from the scalar Schr\"{o}dinger equation
		$$
		L_m v = g - M_m f.
		$$
		Under the non-resonance conditions (\ref{non-resonance}), $0$ is in the spectral gap of $L_m$ making the linear operator $L_m : H^2_{\rm per} \mapsto L^2 $ is invertible with a bounded inverse from $L^2 $ to $H^2_{\rm per}$. Hence $v = L_m^{-1} (g - M_m f)$ and $A_m(\omega_0,c_g)$ is invertible with a bounded inverse from $\tilde{R}$ to $\tilde{D}$.

		The operator $A_1(\omega_0,c_g)$ is not invertible due to the double zero eigenvalue in Lemma \ref{lem-linear}, However, it is a Fredholm operator of  index  zero and
		hence by the closed range theorem there exists a solution of the inhomogeneous equation
		$$
		A_1(\omega_0,c_g) \left( \begin{array}{l} v\\ w \end{array} \right) =  \Pi
		\left( \begin{array}{l} f \\ g \end{array} \right),
		$$
		for every $(f,g) \in \tilde{R}$. The solution is not uniquely determined
		since ${\rm span}(F_0,F_1)$ can be added to the solution $(v,w) \in \tilde{D}$,
		however, the restriction to the subset defined by the condition
		$$
		\Pi   \left( \begin{array}{l} v\\ w \end{array} \right) = \left( \begin{array}{l} v\\ w \end{array} \right)
		$$
		removes projections to ${\rm span}(F_0,F_1)$. Consequently, the operator $(\Pi A_1(\omega_0,c_g) \Pi)$ is invertible with a bounded inverse from $\tilde{R}$ to $\tilde{D}$.
	\end{proof}

	In the limit $\delta \to  0$ of Lemma \ref{lem-constant} we can calculate the eigenvalues $\lambda$ of $A_m(\omega_0,c_g)$ graphically, see Figure \ref{figure4}, and explicitly. The following lemma summarizes the key properties of eigenvalues which are needed for Assumption \ref{manifoldsass} in Theorem \ref{thm1}.

	\begin{figure}
		\centering
	\includegraphics[width=0.5\linewidth]{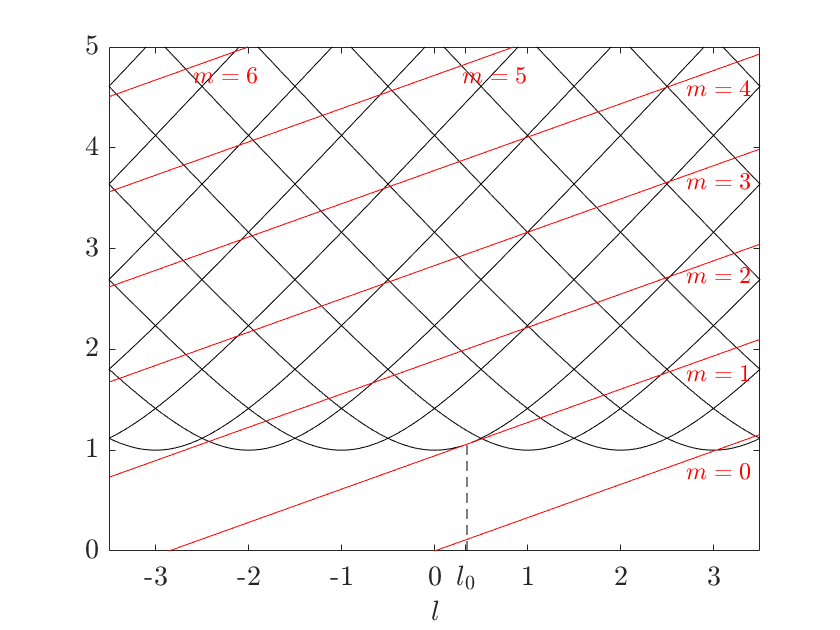}
	\caption{Purely imaginary eigenvalues $ \lambda = \ii (l-ml_0), l\in \RR,$ as roots of the nonlinear equation (\ref{speclin0}) can be obtained graphically as intersections of the curves $ l \mapsto \omega_n(l)$ and $ l \mapsto m\omega_0 +  c_g (l-ml_0)$ for $n\in \mathbb{N}$ and $l \in \mathbb{B}$. For $\rho\equiv 1$ we have $\omega_n(l)=\sqrt{1+(n+l)^2}$
		(not ordered by magnitude) and recall that $\omega_n$ is 1-periodic. Due to the symmetry about the $l$-axis we plot only the upper part. We choose $l_0=0.35$ and $\omega_0=\omega_{1}(l_0)\approx 1.06.$}
	\label{figure4}
\end{figure}

\begin{lemma}
	\label{lem-eigenvalues-constant}
	Let $\rho \equiv 1$. For every fixed $m \in \mathbb{N}_{\rm odd}$, the operator $A_m(\omega_0,c_g)$ has purely imaginary eigenvalues,  Jordan blocks of which have length at most two, and complex semi-simple eigenvalues with nonzero real parts bounded away from zero. Moreover, if $l_0 \in \mathbb{R} \backslash \mathbb{Q}$, then all nonzero, purely imaginary eigenvalues are semi-simple.
\end{lemma}

\begin{proof}
	Eigenvalues of $A_m(\omega_0,c_g)$ are found as solutions of the nonlinear equations \eqref{speclin0}. For $\rho\equiv 1$ we use Fourier series (\ref{Fourier-2}) and write $\omega_n^2(l)=1+(n+l)^2$
	with $n = k \in \mathbb{Z}$. After simple manipulations, eigenvalues $\lambda$ are found from
	\begin{equation*}
		\left( \omega_0^{-1} \lambda + \ii \omega_0 (k-mn_0) \right)^2 =
		1 + k^2 + 2l_0 mk - m^2 (\omega_0^2 - l_0^2) - \omega_0^2 (k - m n_0)^2, \quad k \in \mathbb{Z},
	\end{equation*}
	where $m \in \mathbb{N}_{\rm odd}$, $n_0 \in \mathbb{N}$, and $l_0 \in \mathbb{B}$ are fixed and where $\omega_0 = \sqrt{1 + (n_0+l_0)^2}$.
	Setting $\kappa := k - m n_0$, we get
	\begin{equation*}
		\left( \omega_0^{-1} \lambda + \ii \omega_0 \kappa \right)^2 =
		1 - (m-\kappa(n_0+l_0))^2.
	\end{equation*}
	Eigenvalues $\lambda$ are found explicitly as
	\begin{equation*}
		\lambda = - \ii \kappa \omega_0^2 \pm \ii \omega_0 \sqrt{(m-\kappa(n_0+l_0))^2 - 1}.
	\end{equation*}
	For each $k \in \mathbb{Z}$, the value of $\kappa \in \mathbb{Z}$ is fixed.
	Eigenvalues are double if
	$$
	(m-\kappa(n_0+l_0))^2 - 1 = 0
	$$
	for some $m \in \mathbb{N}_{\rm odd}$ and $\kappa \in \mathbb{Z}$,
	in which case the Jordan blocks have length two.
	If $l_0 \in \mathbb{R} \backslash \mathbb{Q}$ and $\kappa\neq 0$, then $(m-\kappa(n_0+l_0))^2 - 1 \neq 0$ and the eigenvalues are semi-simple, in which case there are no Jordan blocks.

	If $\kappa = 0$, then $\lambda = \pm \ii \omega_0 \sqrt{m^2-1}$ which includes a double zero eigenvalue for $m = 1$ and pairs of semi-simple purely imaginary eigenvalues. If $\kappa \neq 0$, then complex eigenvalues off the imaginary axis arise for each $(m,\kappa)\in \NN_{\rm odd}\times \ZZ$ with $|m - \kappa(n_0 + l_0)| < 1$; the two complex eigenvalues appear in pairs symmetrically about $\ii \mathbb{R}$. Therefore, complex eigenvalues with nonzero real parts are semi-simple. Moreover, ${\rm Im}(\lambda) = -\kappa \omega_0^2$ and hence $|{\rm Im}(\lambda)| \geq \omega_0^2$ for each complex eigenvalue.
\end{proof}

\begin{remark}
	\label{remark-D}
	Conditions (\ref{E:zero_ev_cond2}) ensure that
	$$
	\sqrt{(m-\kappa(n_0+l_0))^2 - 1} \neq \kappa \omega_0, \quad 1 \leq m \leq 2N+1, \quad \kappa \in \NN.
	$$
	As a result, the purely imaginary non-zero eigenvalues of Lemma \ref{lem-eigenvalues-constant} are bounded away from zero by
	$$
	D_m := |\omega_0|\inf_{\kappa \in \NN} | \sqrt{(m-\kappa(n_0+l_0))^2 - 1} - \kappa \omega_0 | > 0, \quad 1 \leq m \leq 2N+1.
	$$
	The inequality $D_m>0$ follows from the fact that
	$$ | \sqrt{(m-\kappa(n_0+l_0))^2 - 1} - \kappa \omega_0 | \sim \kappa(\sqrt{1+(n_0+l_0)^2}-n_0-l_0)$$ as $\kappa\to\infty$, where $\omega_0=\sqrt{1+(n_0+l_0)^2}$ has been used. We do not need invertibility of $A_m(\omega_0,c_g)$ as $m \to \infty$ since we only use the near-identity transformations for $1 \leq m \leq 2N+1$. Therefore, we do not need to investigate whether $D_m \to 0$ as $m \to \infty$.
\end{remark}

In the next two subsections we proceed with near identity transformations by using the bounds (\ref{bound-on-resolvent}) and prove the following theorem.

\begin{theo}
	\label{theorem-transform}
	There exists $\varepsilon_0 > 0$ such that for every $\varepsilon \in (-\varepsilon_0,\varepsilon_0)$, there exists a sequence of near-identity transformations which transforms system (\ref{E:sep_prob}) to the following form:
	\begin{subequations}\label{trans_prob}
		\beq\label{trans_prob_q}
		\left( \begin{array}{c}  				\partial_\xi  q_1\\
			\partial_\xi  q_0 - q_1 \end{array} \right) =
		\sum_{j=1}^{N} \veps^{2j} Z^{(0)}_j(q_0,q_1,S_1^{(N)},{\bf V}^{(N)}_{\geq 3}) + \veps^{2N+2} Z^{(0)}_{N+1}(q_0,q_1,S_1^{(N)},{\bf V}^{(N)}_{\geq 3}),
		\eeq\\
		\beq\label{trans_prob_Y1}
		\partial_\xi  S_1^{(N)} = \Pi A_1(\omega_0,c_g) S_1^{(N)}
		+ \sum_{j=1}^{N} \veps^{2j} Z^{(1)}_j(q_0,q_1,S_1^{(N)},{\bf V}^{(N)}_{\geq 3}) + \veps^{2N+2} Z^{(1)}_{N+1}(q_0,q_1,S_1^{(N)},{\bf V}^{(N)}_{\geq 3}),
		\eeq
		and for $m \in \NN_{\rm odd} \backslash \{1\}$,
		\beq\label{trans_prob_Yk}
		\partial_\xi   Y_m^{(N)} = A_m(\omega_0,c_g) Y_m^{(N)}
		+ \sum_{j=1}^{N} \veps^{2j} Z^{(m)}_j(q_0,q_1,S_1^{(N)},{\bf V}^{(N)}_{\geq 3}) + \veps^{2N+2} Z^{(m)}_{N+1}(q_0,q_1,S_1^{(N)},{\bf V}^{(N)}_{\geq 3}),
		\eeq
		where $Z^{(m)}_j(q_0,q_1,0,{\bf 0}) = 0$ for every $1 \leq j \leq N$ and $m \in\mathbb{N}_{\rm odd}$. The variables $S_1^{(N)}, {\bf V}^{(N)}_{\geq 3}$, and $ Y_m^{(N)}$ are obtained from $S_1, {\bf V}_{\geq 3}$, and $ Y_m$ via $N$ near-identity transformations depending on $q_0$ and $ q_1$, e.g.,
		$$
		S_1^{(N)} = S_1 + \varepsilon^2 \Phi^{(2)}(q_0,q_1) + \varepsilon^4 \Phi^{(4)}(q_0,q_1) + \dots + \varepsilon^{2N} \Phi^{(2N)}(q_0,q_1),
		$$
		where $\Phi^{(2j)}$ depends polynomially on $q_0,\bar{q}_0, q_1,$ and $\bar{q}_1$, and analogously for $ {\bf V}_{\geq 3}^{(N)}$ and $ Y_m^{(N)}$. Moreover, the transformations preserve the reversibility of the system, cf. Section \ref{S:reversib}.
	\end{subequations}
\end{theo}


\begin{remark} \label{rem37}
	It is well known that in the equation for $ u_j $ with eigenvalue $ i \lambda_j $
	a term of the form  $ u_{j_1}^{n_{j_1}} \ldots u_{j_r}^{n_{j_r}} $ can be eliminated by a near identity transformation if the non-resonance condition $  \lambda_j  - \sum_{k=1}^r
	n_{j_k}  \lambda_{n_{j_k}} \neq 0 $ is satisfied, see Sec. 3.3 in \cite{GH83}. Since the eigenvalues for the $ (q_0,q_1) $-part vanish and since the eigenvalues for the $ S_1^{(N)} $-part and the  $ Y_m^{(N)} $-part do not vanish, all  polynomial terms in $  (q_0,q_1) $ can be eliminated
	in the equations for the $ S_1^{(N)} $and $ Y_m^{(N)} $. This elimination is done by Theorem \ref{theorem-transform} up to order $ \mathcal{O}(\veps^{2N}) $.
	Some detailed calculations can be found in the subsequent Sections \ref{sec-34} and \ref{sec-35}.
\end{remark}

\begin{remark} \label{rem38}
	The condition $Z^{(m)}_j(q_0,q_1,0,{\bf 0}) = 0$ for every $1 \leq j \leq N$ and $m \in\mathbb{N}_{\rm odd}$ corresponds to $ \widetilde{N}_r(\widetilde{u}_0,0)=0 $ in system (\ref{cm2}).
	Ignoring the higher order terms
	$$ \veps^{2N+2} Z^{(m)}_{N+1}(q_0,q_1,S_1^{(N)},{\bf V}^{(N)}_{\geq 3}) $$ for $m \in\mathbb{N}_{\rm odd}$
	in \eqref{trans_prob_Y1} and \eqref{trans_prob_Yk}
	gives an invariant subspace $ \{ (S_1^{(N)},{\bf V}^{(N)}_{\geq 3}) = (0,{\bf 0})\} $ in which
	an approximate  homoclinic solution for the full system  can be found.
	It bifurcates for $\varepsilon \neq 0$ from the trivial solution for $\varepsilon = 0$.
\end{remark}

\subsection{Removing polynomial terms in $q_0$ and $q_1$ from (\ref{E:sep_prob_Yk})}
\label{sec-34}

In order to show how the near-identity transformations reduce
(\ref{E:sep_prob_Yk}) to (\ref{trans_prob_Yk}), we consider
a general inhomogeneous term $\veps^{2j} q_0^{k_1} \bar{q}_0^{k_2} q_1^{k_3} \bar{q}_1^{k_4}$ in the right-hand side of (\ref{E:sep_prob_Yk})
with some $1 \leq j \leq N$ and positive integers $k_1, k_2, k_3, k_4$.
The transformations are produced sequentially,
from terms of order  $\mathcal{O}(\veps^2)$ to terms of order $\mathcal{O}(\veps^{2N})$ and for each polynomial order in $(q_0,q_1)$.

At the lowest order $j = 1$, there exists only one inhomogeneous term in (\ref{E:sep_prob_Yk}) for $m = 3$, cf. Remark \ref{rem-1},
which is given by
\begin{align*}
	H_3^{(q)} = -\gamma (1 - c_g^2)^{-1}  \left( \begin{array}{c} 0 \\
		r (q_0 f_{n_0} - \ii q_1 \partial_l f_{n_0})^3 \end{array} \right).
\end{align*}
Hence, $H_3=H_3^{(q)} + H_3^\text{rest}$, where $H_3^\text{rest} = 0$
if $(S_1, {\bf V}_{\geq 3}) = (0,{\bf 0})$. Substituting $Y_3 = \tilde{Y}_3  + \veps^2 \mathfrak{Y}_3$ with
\begin{align*}
	\mathfrak{Y}_3 := h_0 \left( \begin{array}{c} q_0^3 \\
		3 q_0^2 q_1 \end{array} \right) + h_1 \left( \begin{array}{c} q_0^2 q_1 \\
		2 q_0 q_1^2 \end{array} \right) + h_2 \left( \begin{array}{c} q_0 q_1^2 \\
		q_1^3 \end{array} \right)
	+ h_3 \left( \begin{array}{c} q_1^3 \\
		0 \end{array} \right)
\end{align*}
into (\ref{E:sep_prob_Yk}) yields
\begin{align*}
	\partial_\xi   \tilde{Y}_3 = A_3(\omega_0,c_g) \tilde{Y}_3
	+ \veps^2 \widetilde{\omega} (1-c_g^2)^{-1} B_3 \tilde{Y}_3 + \veps^2
	\tilde{H}_3.
\end{align*}
The choice of the second components in each term of $\mathfrak{Y}_3 $ is dictated by the fact that $\partial_\xi q_0-q_1$ and $\partial_\xi q_1$
are of the order of ${\mathcal O}(\varepsilon^2)$ due to equation \eqref{E:sep_prob_q}. We are looking for scalar functions $h_{j} \in H^2_{\rm per}$ from the sequence of linear inhomogeneous equations obtained with the help of (\ref{A-m}):
\begin{align*}
	q_0^3 : & \quad L_3 h_0 = \gamma (1-c_g^2)^{-1} r f_{n_0}^3, \\
	q_0^2 q_1 : & \quad L_3 h_1 = - 3 \ii \gamma (1-c_g^2)^{-1} r f_{n_0}^2 \partial_l f_{n_0} - 3 M_3 h_0, \\
	q_0 q_1^2 : & \quad L_3 h_2 = - 3 \gamma (1-c_g^2)^{-1} r f_{n_0} (\partial_l f_{n_0})^2 - 2 M_3 h_1 + 6 h_0, \\
	q_1^3 : & \quad L_3 h_3 = \ii \gamma (1-c_g^2)^{-1} r (\partial_l f_{n_0})^3 -  M_3 h_2 + 2 h_1.
\end{align*}
Since $L_3$ is invertible with a bounded inverse
by Lemma \ref{lem-invert}, there exist unique functions
$h_{j} \in H^2_{\rm per}$ which are obtained recursively
from $h_0$ to $h_3$. After the inhomogeneous terms are removed by
the choice of $h_{j}$, the transformed right-hand side $\tilde{H}_3 $
becomes
\begin{align*}
	\tilde{H}_3  &= H_3^{{\rm rest}} -h_0 \left( \begin{array}{c} 3 q_0^2 (\dot{q}_0 - q_1) \\
		6 q_0 q_1 (\dot{q}_0 - q_1) + 3 q_0^2 \dot{q}_1 \end{array} \right)
	- h_1 \left( \begin{array}{c} 2 q_0 q_1 (\dot{q}_0 - q_1) + q_0^2 \dot{q}_1\\
		2 (\dot{q}_0 - q_1) q_1^2 + 4 q_0 q_1 \dot{q}_1 \end{array} \right) \\
	& \quad - h_2 \left( \begin{array}{c} (\dot{q}_0 - q_1) q_1^2 + 2 q_0 q_1 \dot{q}_1 \\
		3 q_1^2 \dot{q}_1 \end{array} \right)
	-  h_3 \left( \begin{array}{c} 3 q_1^2 \dot{q}_1 \\
		0 \end{array} \right) + \veps^2 \widetilde{\omega}(1-c_g^2)^{-1} B_3 \mathfrak{Y}_3,
\end{align*}
where $H_3^\text{rest}$ is also modified due to the transformation.
Substituting for $\dot{q}_0-q_1$ and $\dot{q}_1$ from (\ref{E:sep_prob_q}) shows that $\tilde{H}_3(q_0,q_1,0,{\bf 0}) = \mathcal{O}(\veps^2)$, hence the first step of the procedure transforms (\ref{E:sep_prob_Yk}) into (\ref{trans_prob_Yk}) with $N = 1$. One can then define $\tilde{H}_3 = \tilde{H}_3^{(q)} + \tilde{H}_3^\text{rest}$ with  $\tilde{H}_3^\text{rest} = 0$ if $(S_1, {\bf V}_{\geq 3}) = (0,{\bf 0})$
and proceed with next steps of the procedure.

A general step of this procedure is performed similarly. Without loss of generality, since the principal part of system (\ref{E:sep_prob_q}) is independent of $\bar{q}_0$ and $\bar{q}_1$, we consider a general polynomial
of degree $M$ at fixed $m \in \{3,5,\dots,2N+1\}$:
$$
H_m^{(q)} = \sum_{j=0}^{M} q_0^{M-j} q_1^{j} \left( \begin{array}{c} a_j \\
	b_j \end{array} \right),
$$
where $(a_j,b_j)^T$ depend on $x$ only. Substituting
\begin{align*}
	Y_m = \tilde{Y}_m  + \veps^2 \mathfrak{Y}_m, \qquad
	\mathfrak{Y}_m := \sum_{j=0}^{M} q_0^{M-j} q_1^{j} \left( \begin{array}{c} h_j \\
		g_j \end{array} \right)
\end{align*}
into (\ref{E:sep_prob_Yk}) yields (\ref{trans_prob_Yk}) with $N$ being incremented by one if $h_m,g_m \in H^2_{\rm per}$ are found from two chains of recurrence equations for $j \in \{0,1,\dots,M\}$:
\begin{align*}
	g_j &= -a_j + (M+1-j) h_{j-1}, \\
	L_m h_j &= -b_j - M_m g_j + (M+1-j) g_{j-1},
\end{align*}
which are truncated at $h_{-1} = g_{-1} = 0$. Since $L_m$ for $3 \leq m \leq 2N+1$ are invertible with a bounded inverse by Lemma \ref{lem-invert}, the recurrence equations are uniquely solvable from $g_0$ to $h_0$, then to $g_1$ and $h_1$ and so on to $g_M$ and $h_M$. The aforementioned first step is obtained from here with $M = 3$ and $a_0 = a_1 = a_2 = a_3 = 0$.

\subsection{Removing polynomial terms in $q_0$ and $q_1$  from (\ref{E:sep_prob_Y1})}
\label{sec-35}

Similarly, we perform near-identity transformations which reduce (\ref{E:sep_prob_Y1}) to (\ref{trans_prob_Y1}). The only complication is the presence of the projection operator $\Pi$ in system (\ref{E:sep_prob_Y1}).

At  lowest order $j = 1$, there exist two inhomogeneous terms in (\ref{E:sep_prob_Y1}) due to $\Pi B_1 Y_1$ and $\Pi H_1$, which can be written without the projection operator $\Pi$ as follows:
\begin{align*}
	& \frac{\widetilde{\omega}}{1-c_g^2} \left( q_0 B_1 F_0 + q_1 B_1 F_1 \right)
	- \frac{\gamma}{1-c_g^2} \left( \begin{array}{c} 0 \\
		3 r (q_0 f_{n_0} - \ii q_1 \partial_l f_{n_0})^2 (\bar{q}_0 \bar{f}_{n_0} + \ii \bar{q}_1 \partial_l \bar{f}_{n_0}) \end{array} \right) \\
	& =: q_0 {\bf H}^{(0)} + q_1 {\bf H}^{(1)} +
	|q_0|^2 q_0 {\bf H}^{(2)} + q_0^2 \bar{q}_1 {\bf H}^{(3)} + |q_0|^2 q_1 {\bf H}^{(4)}
	+ q_0 |q_1|^2 {\bf H}^{(5)} +  \bar{q}_0 q_1^2 {\bf H}^{(6)} + |q_1|^2 q_1 {\bf H}^{(7)}.
\end{align*}
Substituting $S_1 = \tilde{S}_1  + \veps^2 \mathfrak{S}_1$ with
\begin{align*}
	\mathfrak{S}_1 := q_0 {\bf S}^{(0)} + q_1 {\bf S}^{(1)}
	+ |q_0|^2 q_0 {\bf S}^{(2)}
	+  q_0^2 \bar{q}_1 {\bf S}^{(3)} + |q_0|^2 q_1 {\bf S}^{(4)}
	+ q_0 |q_1|^2 {\bf S}^{(5)} + \bar{q}_0 q_1^2  {\bf S}^{(6)} + |q_1|^2 q_1 {\bf S}^{(7)},
\end{align*}
where $\Pi {\bf S}^{(j)} = {\bf S}^{(j)}$, into (\ref{E:sep_prob_Y1}) yields
\begin{align*}
	\partial_\xi   \tilde{S}_1 = \Pi A_1(\omega_0,c_g) \tilde{S}_1
	+ \veps^2 \widetilde{\omega} (1-c_g^2)^{-1} \Pi B_1 \tilde{S}_1 + \veps^2
	\Pi \tilde{H}_1,
\end{align*}
where $\tilde{H}_1$ is of the next $\mathcal{O}(\veps^2)$ order
if ${\bf S}^{(0)}, {\bf S}^{(1)}$ are chosen from
the system of inhomogeneous equations
\begin{align*}
	q_0 : & \qquad 		\Pi A_1(\omega_0,c_g) {\bf S}^{(0)} = - \Pi {\bf H}^{(0)}, \\
	q_1 : & \qquad 		\Pi A_1(\omega_0,c_g) {\bf S}^{(1)} = -\Pi {\bf H}^{(1)} +  {\bf S}^{(0)},
\end{align*}
and ${\bf S}^{(2)}, \dots, {\bf S}^{(7)}$ are chosen from
the system of inhomogeneous equations
\begin{align*}
	|q_0|^2 q_0 : & \qquad 		\Pi A_1(\omega_0,c_g) {\bf S}^{(2)} = -\Pi
	{\bf H}^{(2)}, \\
	q_0^2 \bar{q}_1 : & \qquad 		\Pi A_1(\omega_0,c_g) {\bf S}^{(3)} = -\Pi
	{\bf H}^{(3)} + {\bf S}^{(2)}, \\
	|q_0|^2 q_1 : & \qquad 		\Pi A_1(\omega_0,c_g) {\bf S}^{(4)} = -\Pi
	{\bf H}^{(4)} + 2 {\bf S}^{(2)}, \\
	q_0 |q_1|^2: & \qquad 	\Pi A_1(\omega_0,c_g) {\bf S}^{(5)} = -\Pi
	{\bf H}^{(5)} + 2 {\bf S}^{(3)} + {\bf S}^{(4)}, \\
	\bar{q}_0^2 q_1^2 : & \qquad 	\Pi A_1(\omega_0,c_g) {\bf S}^{(6)} = -\Pi
	{\bf H}^{(6)} + {\bf S}^{(4)}, \\
	|q_1|^2 q_1 : & \qquad 	\Pi A_1(\omega_0,c_g) {\bf S}^{(7)} = -\Pi
	{\bf H}^{(7)} + {\bf S}^{(5)} + {\bf S}^{(6)}.
\end{align*}
Since $\Pi A_1(\omega_0,c_g) \Pi$ is invertible with a bounded inverse by Lemma \ref{lem-invert}, the two chains of equations are uniquely solvable: from ${\bf S}^{(0)}$ to ${\bf S}^{(1)}$ and from ${\bf S}^{(2)}$ to ${\bf S}^{(7)}$.

It is now straightforward that a general step of the recursive procedure can be performed to reduce (\ref{E:sep_prob_Y1}) to (\ref{trans_prob_Y1}).

\section{Construction of the local center-stable manifold}
\label{sec-4}

By Theorem \ref{theorem-transform}, the spatial dynamical system
can be transformed to the form (\ref{trans_prob}), where the coupling
of $(q_0,q_1) \in \mathbb{C}^2$ with $S_1^{(N)} \in \tilde{D}$ and $Y_m^{(N)} \in \tilde{D}$ for $m \in \NN_{\rm odd} \backslash \{0\}$ occurs at order
$\mathcal{O}(\veps^{2N+2})$. We are now looking for
solutions of system (\ref{trans_prob}) for $\xi$ on $[0,\xi_0]$ for some $\veps$-dependent value $\xi_0 > 0$. In order to produce the bound (\ref{property-2}),
we will need to extend the result to $\xi_0 = \veps^{-(2N+1)}$.

The local center-stable manifold on $[0,\xi_0]$ will be constructed close
to the homoclinic orbit of the system
\begin{equation}
	\label{truncated-system}
	\left( \begin{array}{c} \partial_\xi  q_1\\ \partial_\xi  q_0 - q_1  \end{array} \right) =
	\sum_{j=1}^{N} \veps^{2j} Z^{(0)}_j(q_0,q_1,0,{\bf 0})
\end{equation}
which is a truncation of (\ref{trans_prob_q}). The leading-order term $Z_1^{(0)}(q_0,q_1,0,{\bf 0})$ is computed explicitly as
\begin{align*}
	Z_1^{(0)}(q_0,q_1,0,{\bf 0}) & = \frac{2\widetilde{\omega}}{\omega_0\omega_{n_0}''(l_0)}
	\left( \begin{array}{c} - \omega_0 q_0 +
		\ii ( \omega_0 \langle f_{n_0}, \partial_l f_{n_0} \rangle + c_g) q_1 \\
		\langle \partial_l f_{n_0}-\ii \nu f_{n_0}, f_{n_0} \rangle (\ii\omega_0 q_0 + c_g q_1)
		+  \omega_0 \langle \partial_l f_{n_0}-\ii \nu f_{n_0}, \partial_lf_{n_0} \rangle q_1 \end{array} \right) \\
	& \quad - \frac{3 \gamma}{\omega_0\omega_{n_0}''(l_0)}
	\left( \begin{array}{cc}
		\langle f_{n_0},  r  |q_0 f_{n_0} - \ii q_1 \partial_l f_{n_0} |^2 (q_0 f_{n_0} - \ii q_1 \partial_l f_{n_0}) \rangle \\
		\langle \ii (\partial_l f_{n_0}-\ii \nu f_{n_0}), r  |q_0 f_{n_0} - \ii q_1 \partial_l f_{n_0} |^2 (q_0 f_{n_0} - \ii q_1 \partial_l f_{n_0}) \rangle \end{array} \right).
\end{align*}
The stationary NLS equation (\ref{stat-NLS}) for $A$ rewritten as a first order system for $q_0(\xi) = A(X)$ with $X = \veps \xi$ and $q_1(\xi) = \veps A'(X)$ is
\begin{equation}
	\label{truncated-system-cut}
	\left( \begin{array}{c}
		\partial_\xi  q_1 \\
		\partial_\xi  q_0 - q_1 \end{array} \right) = \veps^2 (\omega_{n_0}''(l_0))^{-1}
	\left( \begin{array}{c}
		-2 \widetilde{\omega} q_0 - \gamma_{n_0}(l_0) |q_0|^2 q_0 \\
		0  \end{array} \right),
\end{equation}
or equivalently
\begin{equation}
	\label{truncated-system-cutA}
	A'' = (\omega_{n_0}''(l_0))^{-1} (-2 \widetilde{\omega} A - \gamma_{n_0}(l_0) |A|^2 A ),
\end{equation}
where $\omega_{n_0}''(l_0) \neq 0$ due to the non-degeneracy condition (\ref{non-degeneracy-2}). Equation \eqref{truncated-system-cut} is the leading order (in $\varepsilon$) part of \eqref{truncated-system} if $q_0= \mathcal{O}(1)$ and $q_1= \mathcal{O}(\varepsilon)$.

The following lemma gives persistence of the sech solution  (\ref{E:NLS-soliton}) of the reduced system \eqref{truncated-system-cut} with $(q_0,q_1) = (A(\varepsilon \cdot), \varepsilon A'(\varepsilon \cdot))$ as a solution of the truncated system (\ref{truncated-system}).

\begin{lemma}
	\label{lemma-homoclinic-orbit}
	Assume (\ref{non-degeneracy-2}) and let $A$ be given by  (\ref{E:NLS-soliton}). For each $N
	\in \mathbb{N}$ and a sufficiently small $\veps>0 $, there exists
	a unique homoclinic orbit of system (\ref{truncated-system})
	satisfying the properties
	\begin{equation}
		{\rm Im}~q_0(0) = 0, \quad {\rm Re}~q_1(0) = 0.
		\label{reversibility-homoclinic-orbit}
	\end{equation}
	such that
	\begin{equation}
		\| q_0 - A(\veps \cdot) \|_{L^{\infty}} \leq C \veps, \quad
		\| q_1 - \veps A'(\veps \cdot) \|_{L^{\infty}} \leq C \veps^2,
		\label{proximity-homoclinic orbit}
	\end{equation}
	and
	\begin{equation}
		|q_0(\xi)| \leq C e^{-\veps \alpha |\xi|}, \quad  |q_1(\xi)| \leq \veps C e^{-\veps \alpha |\xi|}, \quad  \forall \xi \in \mathbb{R},
		\label{bound-homoclinic-orbit}
	\end{equation}
	for some $\alpha > 0$ and $C > 0$.
\end{lemma}

\begin{proof}
	Since
	\begin{equation}
		\label{tech-eq}
		q_0 F_0(x) + q_1 F_1(x) = \left(\begin{array}{cc} q_0 f_{n_0}(l_0,x) - \ii q_1 \partial_l f_{n_0}(l_0,x) \\ q_1 f_{n_0}(l_0,x)
		\end{array} \right),
	\end{equation}
	and since the Fourier coefficients of $f_{n_0}(l_0,x)$ and $\partial_l f_{n_0}(l_0,x)$ are real by Remark \ref{remark-rev},
	the condition (\ref{reversibility-homoclinic-orbit}) expresses the reversibility condition (\ref{reversibility-constraint}) for $m = 1$
	in the linear combination (\ref{tech-eq}).
	The reduced system (\ref{truncated-system-cut}) has two symmetries: if $(q_0(\xi),q_1(\xi))$ is a solution, so is
	\begin{equation}
		\label{symmetry-truncated}
		(q_0(\xi+\xi_0) e^{\ii \theta_0},q_1(\xi+\xi_0) e^{\ii \theta_0})
	\end{equation}
	for real $\xi_0$ and $\theta_0$.
	In the scaling
	$q_0(\xi) = \check{q}_0(X)$ with $X = \veps \xi$ and $q_1(\xi) = \veps \check{q}_1(X)$ system (\ref{truncated-system})  is of the form
	\begin{align}
		\left( \begin{array}{c} \partial_X \check{q}_1\\ \veps^{-1}\left(\partial_X \check{q}_0 - \check{q}_1\right)  \end{array} \right) &  =
		\sum_{j=1}^{j=N} \veps^{2j-2} Z^{(0)}_j(\check{q}_0,\veps \check{q}_1,0,{\bf 0}) \notag \\ & = \left( \begin{array}{c} (\omega_{n_0}''(l_0))^{-1} (-2 \widetilde{\omega} \check{q}_0 - \gamma_{n_0}(l_0) |\check{q}_0|^2 \check{q}_0 ) \\ 0 \end{array} \right) + \mathcal{O}(\veps). 		\label{truncated-system1}
	\end{align}
	For $ \veps = 0 $ there   is a homoclinic orbit $ q_{\rm hom} $ for system \eqref{truncated-system-cut} which is given by $(\check{q}_0,\check{q}_1) = (A,A')$ with $A$ in (\ref{E:NLS-soliton}). The existence of a homoclinic orbit for small $ \veps > 0 $ can be established with the following reversibility argument.
	For $ \veps = 0 $ in the point $ (\check{q}_0^*,\check{q}_1^*) = (\gamma_1,0)$
	with $\gamma_1 = \sqrt{2 |\widetilde{\omega}|  / |\gamma_{n_0}(l_0)|}$,
	the family of homoclinic orbits $ e^{i \theta} q_{\rm hom}(\cdot + \xi) $
	intersects  the fixed space of reversibility
	\begin{equation*}
		{\rm Im}~\check{q}_0(0) = 0, \quad {\rm Re}~\check{q}_1(0) = 0
	\end{equation*}
	transversally. This can be seen as follows, see Figure \ref{fig-homoclinic}.

	In  the coordinates $ (\textrm{Re} \check{q}_0, \textrm{Im} \check{q}_0, \textrm{Re} \check{q}_1, \textrm{Im} \check{q}_1) $ the fixed space of reversibility lies in the span
	of $ (1,0,0,0) $ and $ (0,0,0,1) $. The tangent space at the family of homoclinic orbits $ e^{i \theta} q_{\rm hom}(\cdot + \xi) $ in $ (\check{q}_0^*,\check{q}_1^*) $ is spanned
	by the $ \xi $-tangent vector which is proportional to $ (0,0,1,0) $
	and the $ \theta $-tangent vector which is proportional to $ (0,1,0,0) $.
	Since the vector field of \eqref{truncated-system1}  depends smoothly on the small parameter
	$ 0 < \veps \ll 1$ this intersection persists
	under adding higher order
	terms, i.e. for small $ \veps > 0 $. Thus, the reversibility operator gives a
	homoclinic orbit for \eqref{truncated-system1} for small $ \veps > 0 $, too.
	Undoing the scaling gives the homoclinic orbit
	for the truncated system (\ref{truncated-system}) with
	the exponential decay (\ref{bound-homoclinic-orbit}).

	\begin{figure}[htbp] 
		\centering
		\setlength{\unitlength}{1cm}
		\begin{picture}(12, 5.5)
			\put(3,2){fixed space of reversibility}
			\put(9.5,2){$ \textrm{Re} \ q_0 $}
			\put(1.8,4.5){$ \textrm{Re}  \ q_1 $}
			\put(5,4.9){homoclinic orbit}
			\put(1,2.5){\vector(1,0){9}}
			\put(1.5,0){\vector(0,1){5}}
			\put(1,1){\qbezier(0.4,1.4)(8,6)(8,1)}
		\end{picture}
		\caption{Transversal intersection of the homoclinic orbit of system (\ref{truncated-system}) with the fixed space of the reversibility operator.}
		\label{fig-homoclinic}
	\end{figure}
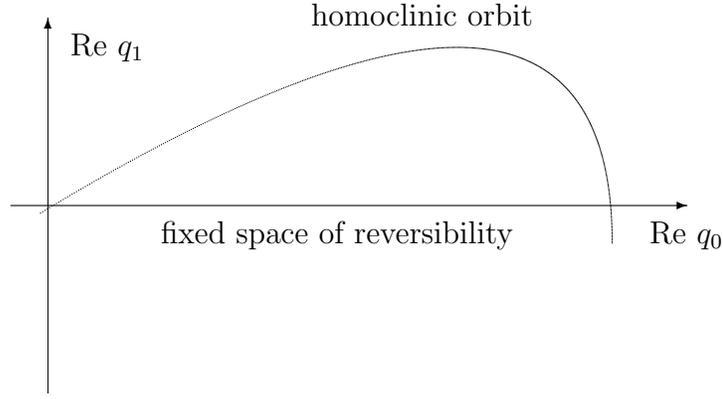

	It remains to prove the approximation bound (\ref{proximity-homoclinic orbit}).
	The symmetry (\ref{symmetry-truncated}) generates the two-dimensional kernel
	of the linearized operator associated with the leading-order part of
	the truncated system (\ref{truncated-system}):
	\begin{equation}
		\label{symmetry-modes}
		\left(\begin{array}{cc} q_0 \\ q_1 \end{array} \right) = \left(\begin{array}{cc} A'(X) \\ \veps A''(X) \end{array} \right) \quad \mbox{\rm and} \quad
		\left(\begin{array}{cc} q_0 \\ q_1 \end{array} \right) = \ii \left(\begin{array}{cc} A(X) \\ \veps A'(X) \end{array} \right)
	\end{equation}
	The symmetry modes (\ref{symmetry-modes}) do not satisfy the reversibility constraints (\ref{reversibility-homoclinic-orbit})
	because $A(0) \neq 0$ and $A''(0) \neq 0$, whereas the truncated system
	(\ref{truncated-system}) inherits the reversibility symmetry (\ref{reversibility-homoclinic-orbit})
	of the original dynamical system (\ref{spat-dyn}). Therefore, if
	we substitute the decomposition
	$$
	\left(\begin{array}{cc} q_0(\xi) \\ q_1(\xi) \end{array} \right)  =
	\left(\begin{array}{cc} A(\veps \xi) \\ \veps A'(\veps \xi) \end{array} \right)
	+ \left(\begin{array}{cc} \mathfrak{q}_0(\xi) \\ \veps \mathfrak{q}_1(\xi) \end{array} \right)
	$$
	into (\ref{truncated-system}), then the correction term
	$(\mathfrak{q}_0,\mathfrak{q}_1)$ satisfies the nonlinear system
	where the residual terms of the order of $\mathcal{O}(\veps)$, see
	(\ref{truncated-system1}), are automatically orthogonal to the kernel
	of the linearized operator. By the implicit theorem in Sobolev space $H^1(\RR)$, one can uniquely solve the nonlinear system for the correction term $(\mathfrak{q}_0,\mathfrak{q}_1)$ under the reversibility
	constraints (\ref{reversibility-homoclinic-orbit}) such that
	$$
	\| \mathfrak{q}_0 \|_{H^1} + \| \mathfrak{q}_1 \|_{H^1} \leq C \veps,
	$$
	for some $\veps$-independent $C > 0$. This yields the
	approximation bound (\ref{proximity-homoclinic orbit}) in the original variables due to the Sobolev embedding of $H^1(\RR)$ into $L^{\infty}(\RR)$.
\end{proof}

\begin{remark}
	The approximation bound (\ref{proximity-homoclinic orbit}) yields the estimate (\ref{property-4}) in Theorem \ref{thm1}, where
	$$
	h(\xi,z,x) = \veps q_0(\xi) f_{n_0}(l_0,x) e^{\ii z} - \ii \varepsilon q_1(\xi) \partial_l f_{n_0}(l_0,x) e^{\ii z} + {\rm c. c.}
	$$
	with $q_0$ and $q_1$ from Lemma \ref{lemma-homoclinic-orbit}.
\end{remark}

\begin{remark}
	Referring to system (\ref{cm1}),  we have now constructed the
	homoclinic solution for the approximate reduced system
	$$
	\partial_{\xi} \widetilde{u}_{0} =  M_{0}\widetilde{u}_{0} + \widetilde{N}_{0}(\widetilde{u}_{0},0).
	$$
	It remains to prove the persistence of the homoclinic solutions as generalized breather solutions under  considering the higher order terms
	$ \veps^{2N+2} Z^{(m)}_{N+1}(q_0,q_1,S_1^{(N)},{\bf V}^{(N)}_{\geq 3}) $
	in \eqref{trans_prob_Y1} and \eqref{trans_prob_Yk}
	for $m \in\mathbb{N}_{\rm odd}$ which lead to
	$(S_1^{(N)},{\bf V}^{(N)}_{\geq 3}) \neq (0,{\bf 0})$.
	We do so by constructing a center-stable manifold nearby the approximate
	homoclinic solution, cf. the rest of Section \ref{sec-4}, and by proving that the center-stable manifold intersects the fixed space of reversibility transversally, cf. Section \ref{sec-5}.
\end{remark}

Let us denote the $\veps$-dependent reversible homoclinic orbit of Lemma \ref{lemma-homoclinic-orbit} by $(Q_0,\veps Q_1)$ and introduce the
decomposition $(q_0,q_1) = (Q_0,\veps Q_1) + (\mathfrak{q}_0,\veps \mathfrak{q}_1)$. We abbreviate  ${\bf c}_{0,hom} := (Q_0,Q_1,\overline{Q_0}, \overline{Q_1})$ and  ${\bf c}_{0,r} := (\mathfrak{q}_0, \mathfrak{q}_1,\overline{\mathfrak{q}_0}, \overline{\mathfrak{q}_1})$. Furthermore, we collect the components $S_1^{(N)}, Y^{(N)}_m$ for $m \in \mathbb{N}_{\rm odd} \backslash \{1\}$ in  ${\bf c}_r$. With these notations system (\ref{trans_prob}) can now be  rewritten
in the abstract form:
\begin{subequations}
	\label{semilinear-4-5}
	\begin{align}
		\label{semilinear-4}
		\partial_\xi  {\bf c}_{0,r} & = \veps
		\Lambda_0(\xi) {\bf c}_{0,r} + \veps {\bf G}({\bf c}_{0,r},{\bf c}_r) +
		\epsilon^{2N+1} {\bf G}_R({\bf c}_{0,hom} + {\bf c}_{0,r},{\bf c}_r), \\
		\label{semilinear-5}
		\partial_\xi  {\bf c}_r & = \Lambda_r(\xi)
		{\bf c}_r + \veps^2 {\bf F}({\bf c}_{0,hom} + {\bf c}_{0,r},{\bf c}_r) + \veps^{2N+2} {\bf F}_R({\bf c}_{0,hom} + {\bf c}_{0,r},{\bf c}_r),
	\end{align}
\end{subequations}
where the vector ${\bf c}_{0,r}(\xi) \in \mathbb{C}^4$ is controlled
in the norm $\| \cdot \|_{\mathbb{C}^4}$, whereas
the vector ${\bf c}_r(\xi)$ is controlled in the phase space $\cD$
defined by (\ref{phase-space}) with the norm $\| \cdot \|_{\cD}$. The operator $ \varepsilon \Lambda_0 $ is the linearization around the homoclinic orbit ${\bf c}_{0,hom}$ and hence  ${\bf G} $ depends nonlinearly on ${\bf c}_{0,r}$. Note that including the complex conjugated variables in ${\bf c}_{0,r}$ is needed in order for the linearized system $\partial_\xi  {\bf c}_{0,r} = \veps\Lambda_0(\xi) {\bf c}_{0,r}$ to be linear with respect to the complex vector field.

\begin{remark}
	At leading order in $\veps$ the matrix $\Lambda_0$ is derived from
	(\ref{truncated-system1}) in the form
	$$
	\Lambda_0 = \begin{pmatrix}
		K & 0 \\ 0 & K
	\end{pmatrix}
	- \gamma_{n_0}(l_0) (\omega_{n_0}''(l_0))^{-1}\begin{pmatrix}
		M_1(A) &	M_2(A)\\
		M_2(A) &	M_1(A)
	\end{pmatrix}
	+ \mathcal{O}(\veps),
	$$
	where
	$$
	K = \begin{pmatrix}
		0 & 1 \\ -2 \widetilde{\omega} (\omega_{n_0}''(l_0))^{-1} & 0
	\end{pmatrix}, \quad
	M_1(A) = \begin{pmatrix}
		0 & 0 \\ 2A^2 & 0
	\end{pmatrix}, \quad
	M_2(A) = \begin{pmatrix}
		0 & 0 \\ A^2 & 0
	\end{pmatrix},
	$$
	using the fact that $A$ is real.
\end{remark}

\begin{remark}
	Although not indicated by our  notation, the operators  $\Lambda_0$, $\Lambda_r$ and the functions ${\bf G}$, ${\bf G}_R$, ${\bf F}$, and ${\bf F}_R$ depend on $\xi$ and $\veps$ continuously.
\end{remark}

\begin{remark}
	We lose one power
	of $\veps$ in front of $ {\bf G} $ and $ {\bf G}_R $ by working with $ Q_1 $ instead of $ \veps Q_1 $.
\end{remark}

\subsection{Residual terms of system (\ref{semilinear-4-5}).} Residual terms  are controlled as follows.

\begin{lemma}
	\label{lemma-residual}
	There exists $\veps_0 > 0$ such that for every $\veps \in (0,\veps_0)$ the residual terms of system (\ref{semilinear-4-5}) satisfy the bounds
	for every $\xi \in \RR$:
	\begin{align*}
		\| {\bf G}({\bf c}_{0,r},{\bf c}_r)(\xi) \|_{\CC^4} &\leq C \left( \| {\bf c}_{0,r}(\xi) \|^2_{\CC^4} + \| {\bf c}_r(\xi) \|_{\cD} \right), \\
		\| {\bf F}({\bf c}_{0,hom} + {\bf c}_{0,r},{\bf c}_r) (\xi)\|_{\cR} &\leq C  \left(  \| {\bf c}_{0,hom}(\xi) + {\bf c}_{0,r}(\xi) \|_{\CC^4} + \| {\bf c}_r(\xi) \|_{\cD}  \right) \| {\bf c}_r
		\|_{\cD}, \\
		\| {\bf G}_R({\bf c}_{0,hom} + {\bf c}_{0,r},{\bf c}_r)(\xi) \|_{\CC^4} &\leq C \left( \| {\bf c}_{0,hom}(\xi) + {\bf c}_{0,r}(\xi) \|_{\CC^4} + \| {\bf c}_r(\xi) \|_{\cD} \right), \\
		\| {\bf F}_R({\bf c}_{0,hom} + {\bf c}_{0,r},{\bf c}_r)(\xi) \|_{\cR} &\leq C \left(  \| {\bf c}_{0,hom}(\xi) + {\bf c}_{0,r}(\xi) \|_{\CC^4} + \| {\bf c}_r(\xi) \|_{\cD}  \right),
	\end{align*}
	as long as $\|  {\bf c}_{0,r}(\xi) \|_{\CC^4} + \| {\bf c}_r(\xi) \|_{\cD} \leq C$,
	where $C > 0$ is a generic $\veps$-independent constant,
	which may change from  line to  line.
\end{lemma}

\begin{proof}
	The residual terms are defined in Theorem \ref{theorem-transform}. Functions ${\bf G}$, ${\bf G}_R$, ${\bf F}$, and ${\bf F}_R$
	map $\cD$ into $\cD$ since
	$\cD$ forms a Banach algebra with respect to pointwise
	multiplication.
	Using $(q_0,q_1) =  {\bf c}_{0,hom} + {\bf c}_{0,r}$  and the fact that $\|{\bf c}_{0,hom}(\xi)\|_{\CC^4}$ is bounded independently of $\veps$, the bounds on ${\bf G}$, ${\bf G}_R$, ${\bf F}$, and ${\bf F}_R$ follow.
\end{proof}

\subsection{Linearized operator of system (\ref{semilinear-4}).}

The linear part of system (\ref{semilinear-4}) is the linearization around the approximate homoclinic orbit $ q_{0,hom} $
from  Lemma \ref{lemma-homoclinic-orbit}. Due to  the translational and the $ SO(2) $-symmetry of the
family of homoclinic orbits generated by these symmetries applied to the reversible homoclinic orbit, the solution space of the linearized equation includes a two-dimensional subspace spanned by exponentially decaying functions.

\begin{lemma}
	\label{lem47}
	Consider the linear inhomogeneous equation
	\begin{equation}
		\label{lin-inhom-eq}
		\partial_\xi  {\bf c}_{0,r} = \veps
		\Lambda_0 {\bf c}_{0,r} + \veps {\bf F}_h,
	\end{equation}
	with a given ${\bf F}_h \in C_b^0(\mathbb{R},\CC^4)$. The homogeneous equation
	has a two-dimensional stable subspace spanned by the two
	fundamental solutions
	\begin{equation}
		\label{kernel} {\bf s}_1(\xi) = {\bf c}_{0,hom}'(\epsilon \xi), \quad {\bf s}_2(\xi) =
		\ii J {\bf c}_{0,hom}(\epsilon \xi),
	\end{equation}
	where $J = {\rm diag}(1,1,-1,-1)$.
	If ${\bf F}_h = (F_0,F_1,\overline{F_0},\overline{F_1})$ satisfies the constraints
	\begin{equation}
		\label{constraints-kernel}
		F_0(\xi) = \overline{F}_0(-\xi), \qquad
		F_1(\xi) = -\overline{F}_1(-\xi), \qquad \xi \in \mathbb{R},
	\end{equation}
	then there exists a two-parameter family of solutions
	${\bf c}_{0,r} \in C_b^0(\mathbb{R})$ in the form
	\begin{equation*}
		{\bf c}_{0,r} = \alpha_1 {\bf s}_1 + \alpha_2 {\bf s}_2 +
		\tilde{\bf c}_{0,r},
	\end{equation*}
	where $(\alpha_1,\alpha_2) \in \CC^2$
	and $\tilde{\bf c}_{0,r} \in C^0_b(\mathbb{R})$ is a particular solution
	satisfying the constraints (\ref{constraints-kernel}) and the bound
	\begin{equation}
		\label{bound-on-zero-subspace}
		\| \tilde{\bf c}_{0,r} \|_{L^{\infty}(\RR)} \leq C
		\| {\bf F}_h \|_{L^{\infty}(\RR)}
	\end{equation}
	for an $\veps$-independent
	constant $C$. \label{lemma-weak-manifolds}
\end{lemma}

\begin{proof}
	As already said,
	the existence of the two-dimensional stable subspace spanned by (\ref{kernel}) follows
	from the translational symmetries due to spatial translations and phase rotations of the truncated system (\ref{truncated-system}).
	Since the truncated system is posed in $\CC^4$, the solution space is  four dimensional and the other two fundamental solutions of the homogeneous equation are
	exponentially growing as $\xi\to\infty$. This can be seen from the limit of $\Lambda_0(\xi)$ as $\xi \to \infty$. Indeed, we have
	$$\lim_{\xi\to\infty}\Lambda_0(\xi)= \begin{pmatrix}
		K & 0 \\ 0 & K
	\end{pmatrix},$$
	the eigenvalues of which are $\pm \sqrt{-2\widetilde{\omega} (\omega_{n_0}''(l_0))^{-1}}$, each being double, where $\widetilde{\omega} (\omega_{n_0}''(l_0))^{-1}<0$ by assumption of Theorem \ref{thm1}.

	As a result, system \eqref{lin-inhom-eq} possesses an exponential dichotomy, see Proposition 1 in Chapter 4 and the discussion starting on page 13 of \cite{coppel1971dichotomies}. The existence of a particular bounded solution ${\bf c}_{0,r} \equiv {\bf c}_{0,r}^{(p)}$ satisfying the bound  \eqref{bound-on-zero-subspace} now follows from Theorem 7.6.3 in \cite{henry1981}. Let us then define $\tilde{\bf c}_{0,r} := {\bf c}_{0,r}^{(p)} + \tilde{\alpha}_1 {\bf s}_1  + \tilde{\alpha}_2 {\bf s}_2 = (\tilde{\mathfrak{q}}_0,\tilde{\mathfrak{q}}_1, \overline{\tilde{\mathfrak{q}}_0},\overline{\tilde{\mathfrak{q}}_1})$ and pick the unique values of $\tilde{\alpha}_1$ and $\tilde{\alpha}_2$ to satisfy the constraints
	\begin{equation}
		\label{constraint-c}
		{\rm Im}(\tilde{\mathfrak{q}}_0)(0) = 0 \quad \mbox{\rm and } \quad
		{\rm Re}(\tilde{\mathfrak{q}}_1)(0) = 0.
	\end{equation}
	This is always possible since
	$$
	{\bf s}_1(0) = \left( \begin{array}{c} Q_0'(0) \\ Q_1'(0) \\ \bar{Q}_0'(0) \\ \bar{Q}_1'(0) \end{array} \right) \quad \mbox{\rm and} \quad	{\bf s}_2(0) = \left( \begin{array}{c} \ii Q_0(0) \\ \ii Q_1(0) \\ -\ii \bar{Q}_0(0)  \\ -\ii \bar{Q}_1(0) \end{array} \right),
	$$
	where ${\rm Re}(Q_0(0)) = A(0) + \mathcal{O}(\veps)$,
	${\rm Im}(Q_1(0)) = \mathcal{O}(\veps)$, ${\rm Im}(Q_0'(0)) = \mathcal{O}(\veps)$, and ${\rm Re}(Q_1'(0)) = A''(0) + \mathcal{O}(\veps)$ by Lemma \ref{lemma-homoclinic-orbit} with $A(0) \neq 0$ and $A''(0) \neq 0$. Hence, for every ${\bf c}_{0,r}^{(p)} \in L^{\infty}(\RR)$, the linear system (\ref{constraint-c}) for $\tilde{\alpha}_1$ and $\tilde{\alpha}_2$ admits a unique solution such that
	$$
	|\tilde{\alpha}_1| + |\tilde{\alpha}_2| \leq C \| {\bf c}_{0,r}^{(p)} \|_{L^{\infty}(\RR)}.
	$$
	The solution $\tilde{\bf c}_{0,r} = {\bf c}_{0,r}^{(p)} + \tilde{\alpha}_1 {\bf s}_1  + \tilde{\alpha}_2 {\bf s}_2$ is bounded and satisfies the bound (\ref{bound-on-zero-subspace}).

	The matrix $\Lambda_0$ commutes with the symmetry operator defined by \eqref{constraints-kernel}, i.e. if ${\bf f}\in C^0_b(\RR,\CC^4)$ satisfies \eqref{constraints-kernel}, then so does $\Lambda_0 {\bf f}$. In addition, the right-hand side ${\bf F}_h = (F_0,F_1,\overline{F_0},\overline{F_1})$ satisfies \eqref{constraints-kernel}. Hence the vector field is closed
	in the subspace satisfying (\ref{constraints-kernel}). If
	a bounded solution $\tilde{\bf c}_{0,r} = (\tilde{\mathfrak{q}}_0,\tilde{\mathfrak{q}}_1, \overline{\tilde{\mathfrak{q}}_0},\overline{\tilde{\mathfrak{q}}_1})$ on $(0,\infty)$ satisfies the constraints (\ref{constraint-c}), then
	its extension on $(-\infty,\infty)$ belongs to the subspace satisfying
	\eqref{constraints-kernel}. Thus, the existence of the bounded solution $\tilde{\bf c}_{0,r}$ of (\ref{lin-inhom-eq}) satisfying (\ref{constraints-kernel}) and (\ref{bound-on-zero-subspace}) is proven.
	A general bounded solution of (\ref{lin-inhom-eq})
	has the form ${\bf c}_{0,r} = \alpha_1 {\bf s}_1 +
	\alpha_2 {\bf s}_2 + \tilde{\bf c}_{0,r}$, where
	$(\alpha_1,\alpha_2)\in\CC^2$ are arbitrary.
\end{proof}

\begin{remark}
	If $|\alpha_1|+ |\alpha_2| \neq 0$, then the solution ${\bf c}_{0,r} = \alpha_1 {\bf s}_1 + \alpha_2 {\bf s}_2 + \tilde{\bf c}_{0,r}$ does not satisfy the reversibility constraint (\ref{constraints-kernel}) because
	${\bf s}_1$ and ${\bf s}_2$ violate the reversibility constraints.
\end{remark}

\subsection{Estimates for the local center-stable manifold.}

We are now ready to construct a local center-stable manifold
for system (\ref{semilinear-4-5}). Let us split the components in ${\bf c}_r$ in three sets denoted
by ${\bf c}_s$, ${\bf c}_u$, and ${\bf c}_c$, where
${\bf c}_s$, ${\bf c}_u$, and ${\bf c}_c$ correspond to components of $\Lambda_r$ with eigenvalues $\lambda$ with ${\rm Re}(\lambda) < 0$, ${\rm Re}(\lambda) > 0$, and ${\rm Re}(\lambda) = 0$ respectively.

\begin{remark}
	These coordinates correspond to the stable, unstable, and reduced center manifold of the linearized system in Lemma \ref{lem-linear}, where the reduced center manifold is obtained after the double zero eigenvalue is removed since the eigenspace of the double zero eigenvalue is represented by the coordinate ${\bf c}_{0,r}$.
\end{remark}

We study the coordinates ${\bf c}_s$, ${\bf c}_u$, and ${\bf c}_c$
in  subsets of the phase space $\cD$ denoted by $\cD_s$, $\cD_u$, and $\cD_c$ respectively. Similarly, the restrictions of $\Lambda_r$ to the three subsets of $\cD$ are denoted by $\Lambda_s$, $\Lambda_u$, and $\Lambda_c$ respectively. Moreover, let  $P_j$ for $ j = s,u,c $ be the projection operator from $\cD$ to $\cD_j$ satisfying
$\|P_s	\|_{\cD \to \cD}+ \|P_u	\|_{\cD \to \cD} + \|P_c	\|_{\cD \to \cD} \leq C$ for some $C > 0$.

We make the following assumption on the semi-groups generated by the linearized system, cf. \cite{LBCCS09}.

\begin{ass}
	\label{manifoldsass}
	There exist $ K > 0$ and $ \veps_0 > 0 $ such that for all $ \veps \in (0, \veps_0) $ we have
	\begin{align*}
		\|e^{\Lambda_s \xi} \|_{\cD \to \cD}
		& \leq K , \qquad  \xi \geq 0,
		\\
		\|e^{\Lambda_u \xi}	\|_{\cD \to \cD}&  \leq K ,  \qquad
		\xi \leq 0,
		\\
		\|e^{\Lambda_c \xi}	\|_{\cD \to \cD} & \leq K , \qquad   \xi \in \mathbb{R}.
	\end{align*}
\end{ass}

The following theorem gives the construction of the local center-stable manifold near the reversible homoclinic orbit of Lemma \ref{lemma-homoclinic-orbit}. It also provides a classification of  all parameters of the local manifold which will be needed in Section \ref{sec-5} to satisfy the reversibility conditions. The center-stable manifold is constructed for $  \xi \in [0,\veps^{-(2N+1)}] $ and not for all $ \xi \geq 0 $. The bound of $\mathcal{O}(\veps^{2N})$ on the coordinates ${\bf c}_{0,r}$ and ${\bf c}_r$
is consistent with the bound (\ref{property-2}) in Theorem \ref{thm1}.

\begin{theo}
	\label{theorem-reduction}
	Under Assumption \ref{manifoldsass}, there exist $ \veps_0 > 0 $, $ C > 0 $ such that for all
	$\veps \in (0,\veps_0)$
	the following holds.
	For every ${\bf a} \in \cD_c$, ${\bf b} \in \cD_s$, and $(\alpha_1,\alpha_2) \in \CC^2$  satisfying
	\begin{equation}
		\label{bound-initial}
		\|{\bf a} \|_{\cD_c} \leq C \veps^{2N},
		\quad \| {\bf b} \|_{\cD_s} \leq C \veps^{2N}, \quad |\alpha_1| + |\alpha_2| \leq C \veps^{2N},
	\end{equation}
	there exists a family of  local solutions of system (\ref{semilinear-4-5}) satisfying the bound
	\begin{equation}
		\label{bound-center} \sup_{ \xi \in [0,\veps^{-(2N+1)}]} ( \|
		{\bf c}_{0,r}(\xi) \|_{\CC^4} + \| {\bf c}_c(\xi) \|_{\cD_c} + \| {\bf c}_s(\xi) \|_{\cD_s} +  \| {\bf c}_u(\xi) \|_{\cD_u}
		)\leq C
		\veps^{2N},
	\end{equation}
	as well as the identities ${\bf c}_c(0) = {\bf a}$,
	$e^{-\xi_0 \Lambda_s} {\bf c}_s(\xi_0) = b$ at $\xi_0 = \veps^{-(2N+1)}$,
	and ${\bf c}_{0,r} = \alpha_1 {\bf s}_1 + \alpha_2 {\bf s}_2 + \tilde{\bf c}_{0,r}$ with uniquely defined $\tilde{\bf c}_{0,r}:[0,\veps^{-(2N+1)}] \to\CC^4$.
\end{theo}

\begin{proof}
	In order to construct solutions of system (\ref{semilinear-4-5}) on $[0,\xi_0]$ with some $\veps$-dependent $\xi_0 > 0$, we multiply the nonlinear vector field of system
	(\ref{semilinear-5}) by a smooth  cut-off function $\chi_{[0,\xi_0]} \in C^{\infty}([0,\infty))$ such that
	\begin{equation}
		\label{semilinear-6}
		\partial_\xi  {\bf c}_r = \Lambda_r
		{\bf c}_r + \veps^2 \chi_{[0,\xi_0]} {\bf F}({\bf c}_{0,hom} + {\bf c}_{0,r},{\bf
			c}_{r}) + \veps^{2N+2} \chi_{[0,\xi_0]} {\bf F}_R({\bf c}_{0,hom} + {\bf
			c}_{0,r},{\bf c}_r), \quad \xi \in [0,\xi_0],
	\end{equation}
	where $\chi_{[0,\xi_0]}(\xi) =  1 $  for $ \xi \in [0,\xi_0]$
	and $\chi_{[0,\xi_0]}(\xi) = 0 $ for $\xi \in (\xi_0,\infty)$.
	Similarly, we multiply the nonlinear vector field of
	system (\ref{semilinear-4}) by the same cut-off function
	and add a symmetrically reflected vector field on $[-\xi_0,0]$ to obtain
	\begin{align}
		\nonumber
		\partial_\xi  {\bf c}_{0,r} & = \veps
		\Lambda_0 {\bf c}_{0,r} + \veps \chi_{[0,\xi_0]} {\bf G}({\bf c}_{0,r},{\bf c}_r) +
		\epsilon^{2N+1} \chi_{[0,\xi_0]} {\bf G}_R({\bf c}_{0,hom} + {\bf
			c}_{0,r},{\bf c}_r) \\
		\nonumber
		& + \veps \chi_{[-\xi_0,0]} {\bf G}^*({\bf
			c}_{0,r},{\bf c}_r) +
		\epsilon^{2N+1} \chi_{[-\xi_0,0]} {\bf G}_R^*({\bf c}_{0,hom} + {\bf
			c}_{0,r},{\bf c}_r), \quad \xi \in [-\xi_0,\xi_0]\\
		&=:\veps \Lambda_0  {\bf c}_{0,r} + \veps \tilde{\bf G}( {\bf c}_{0,r}, {\bf c}_{r}) +\veps^{2N+1}\tilde{\bf G}_r,
		\label{semilinear-7}
	\end{align}
	where
	$$
	\bar{\bf G}_0^*(-\xi) := {\bf G}_0(\xi), \quad
	\bar{\bf G}_1^*(-\xi) := -{\bf G}_1(\xi), \quad
	\bar{\bf G}_{R,0}^*(-\xi) := {\bf G}_{R,0}(\xi), \quad
	\bar{\bf G}_{R,1}^*(-\xi) := -{\bf G}_{R,1}(\xi),
	$$
	for all $\xi \in [0,\xi_0]$, resulting in $\tilde{\bf G}$ and $\tilde{\bf G}_R$ satisfying the reversibility condition \eqref{constraints-kernel}. This modification allows us to apply Lemma \ref{lemma-weak-manifolds} on $\mathbb{R}$.

	We are looking for a global solution of system (\ref{semilinear-6})-(\ref{semilinear-7}) for $\xi \in [0,\infty) $ which may be unbounded as $\xi \to \infty$. This global solution coincides with a local solution of system (\ref{semilinear-4-5}) on the interval
	$[0,\xi_0] \subset \mathbb{R}$.

	We write ${\bf c}_{0,r} = \alpha_1 {\bf s}_1 + \alpha_2 {\bf s}_2 + \tilde{\bf c}_{0,r}$ and rewrite \eqref{semilinear-7} as an equation for $\tilde{\bf c}_{0,r}$. By the construction of the vector field in system (\ref{semilinear-7}), the vector field satisfies the reversibility constraints (\ref{constraints-kernel}). By the bounds of Lemma \ref{lemma-residual} and
	the invertibility of the linear operator in Lemma \ref{lemma-weak-manifolds},  the implicit function theorem implies that there exists a unique map from ${\bf c}_r \in C^0_b([0,\xi_0],\cD)$ to ${\bf c}_{0,r} \in C^0_b([0,\xi_0],\CC^4)$ satisfying
	\begin{equation}
		\label{mapping-stable-unstable-weak} \sup_{ \xi \in [0,\xi_0]}
		\| {\bf c}_{0,r}(\xi) \|_{\CC^4} \leq |\alpha_1| + |\alpha_2| + C
		\sup_{ \xi \in [0,\xi_0]} \| {\bf c}_r(\xi) \|_{\cD} +
		\veps^{2N} C \sup_{\xi \in [0,\xi_0]} \left(  1 + \| {\bf
			c}_r(\xi) \|_{\cD} \right),
	\end{equation}
	as long as $|\alpha_1| + |\alpha_2| + \sup\limits_{\xi \in [0,\xi_0]} \| {\bf c}_r(\xi) \|_{\cD} \leq C \veps^{\mu}$ for some $C > 0$ and $\mu > 0$.

	Using the variation of constant formula the solution of system (\ref{semilinear-6}) projected to $\cD_c \oplus \cD_s \oplus \cD_u$ can be rewritten in the integral form
	\begin{align}
		\nonumber {\bf c}_c(\xi) &= e^{\xi \Lambda_c} {\bf a} +
		\veps^2 \int_0^{\xi}  e^{(\xi-\xi') \Lambda_c} P_c  {\bf F}({\bf
			c}_{0,hom}(\epsilon \xi') + {\bf c}_{0,r}(\xi'),{\bf c}_r(\xi')) d \xi'\\
		\label{mapping-center}
		&	 + \veps^{2N+2} \int_0^{\xi}  e^{(\xi-\xi') \Lambda_c} P_c  {\bf
			F}_R({\bf c}_{0,hom}(\epsilon \xi') + {\bf c}_{0,r}(\xi'),{\bf c}_r(\xi'))	d\xi',
	\end{align}
	\begin{align}
		\nonumber {\bf c}_s(\xi) &= e^{\xi \Lambda_s} {\bf b} -
		\veps^2 \int_{\xi}^{\xi_0}  e^{(\xi-\xi') \Lambda_s} P_s  {\bf F}({\bf
			c}_{0,hom}(\epsilon \xi') + {\bf c}_{0,r}(\xi'),{\bf c}_r(\xi')) d \xi'\\
		\label{mapping-center1}
		&	 - \veps^{2N+2} \int_{\xi}^{\xi_0}  e^{(\xi-\xi') \Lambda_s} P_s  {\bf
			F}_R({\bf c}_{0,hom}(\epsilon \xi') + {\bf c}_{0,r}(\xi'),{\bf c}_r(\xi'))	d\xi',
	\end{align}
	and
	\begin{align}
		\nonumber {\bf c}_u(\xi) &=
		-\veps^2 \int_{\xi}^{\xi_0}  e^{(\xi-\xi') \Lambda_u} P_u  {\bf F}({\bf
			c}_{0,hom}(\epsilon \xi') + {\bf c}_{0,r}(\xi'),{\bf c}_r(\xi')) d \xi'\\
		\label{mapping-center2}
		&	 - \veps^{2N+2} \int_{\xi}^{\xi_0}  e^{(\xi-\xi') \Lambda_u} P_u  {\bf
			F}_R({\bf c}_{0,hom}(\epsilon \xi') + {\bf c}_{0,r}(\xi'),{\bf c}_r(\xi'))	d\xi',
	\end{align}
	where ${\bf c}_c(0) = {\bf a}$, ${\bf c}_s(\xi_0) = e^{\xi_0 \Lambda_s}  {\bf b}$, and ${\bf c}_u(\xi_0) = {\bf 0}$. It is assumed in (\ref{mapping-center}), (\ref{mapping-center1}), and (\ref{mapping-center2}) that
	${\bf c}_{0,r} \in C^0_b([0,\xi_0],\CC^4)$ is expressed in terms of
	${\bf c}_r \in C^0_b([0,\xi_0],\cD)$ by using the map satisfying (\ref{mapping-stable-unstable-weak}). The existence of a unique local (small) solution ${\bf c}_c \in C^0_b[0,\xi_0],\cD_c)$, ${\bf c}_s \in C^0_b([0,\xi_0],\cD_s)$, and ${\bf c}_u \in C^0_b([0,\xi_0],\cD_u)$
	in the system of integral equations  (\ref{mapping-center}), (\ref{mapping-center1}), and (\ref{mapping-center2})  follows from
	the implicit function theorem for small $\veps>0$ and finite $\xi_0 > 0$. To estimate this solution and to continue it for larger values of $\xi_0$, we use the bounds of Lemma \ref{lemma-residual} and Assumption \ref{manifoldsass}. It follows from (\ref{mapping-center}) that
	\begin{align*}
		\sup_{\xi \in [0,\xi_0]} \| {\bf c}_c(\xi) \|_{\cD_c} &\leq K \left[
		\| {\bf a} \|_{\cD_c} +
		\veps^2 C \int_0^{\xi_0} \| {\bf c}_{0,hom}(\epsilon \xi) \|_{\CC^4} \| {\bf
			c}_r(\xi)
		\|_{\cD} d\xi + \veps^2 C \xi_0 \sup_{\forall \xi \in [0,\xi_0]} \| {\bf
			c}_r(\xi) \|^2_{\cD} \right. \\
		&+ 		\left.  \veps^{2N+2} C \int_0^{\xi_0} \| {\bf
			c}_{0,hom}(\epsilon \xi) \|_{\CC^4} d\xi + \veps^{2N+2} C \xi_0  \sup_{
			\xi \in [0,\xi_0]} \| {\bf c}_r(\xi) \|_{\cD} \right]
	\end{align*}
	as long as $|\alpha_1| + |\alpha_2| + \sup\limits_{\xi \in [0,\xi_0]} \| {\bf c}_r(\xi) \|_{\cD} \leq C \veps^{\mu}, \mu >0$. Similar estimates are obtained for $ \sup\limits_{\xi \in [0,\xi_0]} \| {\bf c}_s(\xi) \|_{X_s} $ and $ \sup\limits_{\xi \in [0,\xi_0]} \| {\bf c}_u(\xi) \|_{X_u} $.

	We denote
	$$
	S(\xi_0) := \sup_{ \xi \in [0,\xi_0]} \| {\bf c}_c(\xi) \|_{\cD_c}
	+ \sup_{\xi \in [0,\xi_0]} \| {\bf c}_s(\xi) \|_{\cD_s}
	+  \sup_{ \xi \in [0,\xi_0]} \| {\bf c}_u(\xi) \|_{\cD_u}.
	$$
	Since
	$$
	\epsilon \int_0^{\infty} \| {\bf c}_{0,hom}(\epsilon \xi) \|_{\CC^4} d\xi  \leq C
	$$
	due to the bound (\ref{bound-homoclinic-orbit}),
	it follows from the previous estimates that there exists $C > 0$ such that
	\begin{align}
		\nonumber
		S(\xi_0)
		&
		\leq C \left( \| {\bf a} \|_{\cD_c} + \|{\bf b}\|_{\cD_s} + \veps^{2N+1} +
		\veps S(\xi_0)
		+ \veps^2 \xi_0 S(\xi_0)^2 	+  \veps^{2N+2} \xi_0 S(\xi_0)  \right).
	\end{align}
	Using a bootstrapping argument, we show that $S(\veps^{-(2N+1)}) \leq C \veps^{2N}$ if $\veps>0$ is small enough. To do so, let us choose ${\bf a}$, ${\bf b}$ and $(\alpha_1,\alpha_2)$ to satisfy the
	bound (\ref{bound-initial}) and let $\delta \in (0,1)$. Then
	\begin{align}
		S(\xi_0) \leq C\left( \veps^{2N} + (\veps +\veps^{2N+1+\delta} \xi_0) S(\xi_0)  \right)
		\label{center-manifold-bound}
	\end{align}
	as long as $S(\xi_0)\leq \veps^{2N-1+\delta}.$

	For $\xi_0=0$ we have $S(\xi_0)\leq C\veps^{2N}$ because $\|{\bf c}_c(0)\|_{\cD_c}=\|{\bf a}\|_{\cD_c}$, $\|{\bf c}_s(0)\|_{\cD_s}=\|{\bf b}\|_{\cD_s}$, and $\|{\bf c}_u(0)\|_{\cD_u}=0$. Let us assume that there is $\xi_*\in (0, \veps^{-(2N+1)}]$ such that $S(\xi_*)=\veps^{2N-1+\delta}$ and $S(\xi_0)<\veps^{2N-1+\delta}$ for all $\xi_0\in (0,\xi_*)$. Then \eqref{center-manifold-bound} implies $S(\xi_0)\leq C(\veps^{2N} +\veps^\delta S(\xi_0))$ for all $\xi_0 \in (0,\xi_*)$ and hence
	$$S(\xi_*)\leq C\veps^{2N} < \veps ^{2N-1+\delta}$$
	for $\veps>0$ small enough. This is a contradiction and we get that  $S(\xi_0)\leq \veps^{2N-1+\delta}$ for all $\xi_0\in [0, \veps^{-(2N+1)}]$. Applying again \eqref{center-manifold-bound}, we get
	\begin{equation}
		\label{bound-final} S(\veps^{-(2N+1)}) \leq C \veps^{2N}.
	\end{equation}
	In view of the bound (\ref{mapping-stable-unstable-weak}), it follows that
	the local solution satisfies the bound (\ref{bound-center}).
\end{proof}

\begin{remark}
	The proximity bound (\ref{bound-center}) yields the estimate (\ref{property-2}) in Theorem \ref{thm1}.
\end{remark}

\begin{remark}
	Assumption \ref{manifoldsass} can be satisfied
	for smooth small-contrast potentials, see Lemmas \ref{lem-constant} and \ref{lem-eigenvalues-constant}.
	For $ \rho $ with a small non-zero contrast spectral gaps occur in Figure \ref{figure4}.	Smoothness of $ \rho $ allows to control the size of the spectral gaps for large $ \lambda $, cf.  \cite{eastham}. Assumption \ref{manifoldsass} can be weakened and Jordan-blocks can be allowed, see Remark \ref{R:Jord-2}.
\end{remark}

\begin{remark}\label{R:Jord-2}
	In the generic case of eigenvalues, the Jordan blocks of which have length two, the bounds of Assumption \ref{manifoldsass} must be replaced by
	\begin{align*}
		\|e^{\Lambda_s \xi} \|_{\cD \to \cD}
		& \leq K |t|, \qquad  \xi \geq 0,
		\\
		\|e^{\Lambda_u \xi}	\|_{\cD \to \cD}&  \leq K |\xi| ,  \qquad
		\xi \leq 0,
		\\
		\|e^{\Lambda_c \xi}	\|_{\cD \to \cD} & \leq K |\xi|, \qquad
		\xi \in \mathbb{R}.
	\end{align*}
	The equivalent bound for the estimate of $\sup_{\xi \in [0,\xi_0]} \| {\bf c}_c(\xi) \|_{\cD_c}$ is given by
	\begin{align*}
		\sup_{\xi \in [0,\xi_0]} \| {\bf c}_c(\xi) \|_{\cD_c} &\leq K \left[
		\xi_0 \| {\bf a} \|_{\cD_c} +
		\veps^2 C \int_0^{\xi_0} (\xi_0 -\xi)  \| {\bf c}_{0,hom}(\epsilon \xi) \|_{\CC^4} \| {\bf
			c}_r(\xi)
		\|_{\cD} d\xi \right. \\
		&+ 		\left. \quad \veps^2 C \xi_0^2 \sup_{\xi \in [0,\xi_0]} \| {\bf
			c}_r(\xi) \|^2_{\cD}  +  \veps^{2N+2} C \int_0^{\xi_0} (\xi_0 - \xi)\| {\bf
			c}_{0,hom}(\epsilon \xi) \|_{\CC^4} dy \right. \\
		&+ \left. \quad \veps^{2N+2} C \xi_0^2  \sup_{
			\xi \in [0,\xi_0]} \| {\bf c}_r(\xi) \|_{\cD} \right]
	\end{align*}
	and similarly for
	$ \sup\limits_{\xi \in [0,\xi_0]} \| {\bf c}_s(\xi) \|_{X_s} $ and
	$ \sup\limits_{\xi \in [0,\xi_0]} \| {\bf c}_u(\xi) \|_{X_u}  $.
	This yields with the help of Gronwall's inequality and the bound
	$$
	\lim_{\xi_0 \to \infty} \epsilon^2 \int_0^{\xi_0} (\xi_0 - \xi) \| {\bf c}_{0,hom}(\epsilon \xi) \|_{\CC^4} d\xi
	= \lim_{y_0 \to \infty} \int_0^{y_0} (y_0 - y) \| {\bf c}_{0,hom}(y) \|_{\CC^4} d y  \leq C
	$$
	that
	\begin{align}
		\nonumber
		S(\xi_0)
		&
		\leq C \left( \xi_0 \| {\bf a} \|_{\cD_c} + \xi_0 \|{\bf b}\|_{\cD_s} + \veps^{2N} + \veps^2 \xi_0^2 S(\xi_0)^2 + \veps^{2N+2} \xi_0^2 S(\xi_0)  \right).
		\label{center-manifold-bound-new}
	\end{align}
	This bound still implies $S(\xi_0) \leq C \veps^{2N}$
	but for $\xi_0 = \veps^{-N-1}$ and  $\| {\bf a} \|_{\cD_c} + \|{\bf b}\|_{\cD_s} \leq \veps^{3N+1}$. Thus, the justification result of Theorem \ref{thm1} can be extended on the scale of $\xi \in \mathcal{O}(\veps^{-N-1})$ to the generic case of eigenvalues with Jordan blocks of length two when Assumption \ref{manifoldsass} cannot be used, see Lemma \ref{lem-eigenvalues-constant}.
\end{remark}

\section{End of the proof of Theorem \ref{thm1}}
\label{sec-5}

In Theorem \ref{theorem-reduction} we  constructed a family of local bounded solutions of system (\ref{semilinear-4-5}) on
$[0,\veps^{-(2N+1)}]$. These solutions are close to the reversible homoclinic orbit of Lemma \ref{lemma-homoclinic-orbit} in the sense of the bound \eqref{property-2} for appropriately defined $v$ and $h$ but only on $[0,\veps^{-(2N+1)}]$. It remains to extract those solutions of this family which satisfy  \eqref{property-2}  not only on $[0,\veps^{-(2N+1)}]$, but also on $[-\veps^{-(2N+1)},\veps^{-(2N+1)}]$. We do so by
extending  the local solutions on $[0,\veps^{-(2N+1)}]$
to the interval $[-\veps^{-(2N+1)},\veps^{-(2N+1)}]$
with the help of the reversibility constraints.
Obviously this is only possible for the solutions 	which intersect
the fixed space of reversibilty.
Hence, for the proof of Theorem \ref{thm1} it remains to prove that
the local invariant center-stable  manifold
of system (\ref{semilinear-4-5}) intersects the
subspace given by
the reversibility constraints (\ref{reversibility-constraint}).

%
	%
	%

\begin{itemize}
\item Since the initial data ${\bf c}_c(0) = {\bf a} \in \cD_c$ in the
local center-stable manifold of Theorem \ref{theorem-reduction}
are arbitrary, the components of ${\bf a}$ can be chosen to satisfy
the reversibility constraints (\ref{reversibility-constraint}).
For example, using $\hat{a}_{m,k} = (\hat{a}_{m,k}^{(v)},\hat{a}_{m,k}^{(w)})$
for Fourier representation (\ref{Fourier-2}), we can specify the
reversibility constraints as
$$
{\rm Im} \; \hat{a}_{m,k}^{(v)} = 0, \quad
{\rm Re} \; \hat{a}_{m,k}^{(w)} = 0, \quad
(m,k) \in \NN_{\rm odd} \times \ZZ.
$$
\vspace{0.1cm}

\item We have ${\bf c}_{0,r} = \alpha_1 {\bf s}_1 + \alpha_2 {\bf s}_2 +
\tilde{\bf c}_{0,r}$, where $\tilde{\bf c}_{0,r}$ satisfies the reversibility constraints (\ref{reversibility-constraint}) by Lemma \ref{lem47}.
Since  ${\bf s}_1$ and ${\bf s}_2$ violate
(\ref{reversibility-constraint}), setting $\alpha_1 = \alpha_2 = 0$
satisfies the reversibility constraints for ${\bf c}_{0,r} = \tilde{\bf c}_{0,r}$. The choice of $\alpha_1 = \alpha_2 = 0$ is unique by the implicit
function theorem in the proof of Theorem \ref{theorem-reduction}.\\

\item The initial data ${\bf c}_s(0)$ and ${\bf c}_u(0)$ are not arbitrary since the stable and unstable manifold theorems are used for construction of ${\bf c}_{s}$ and ${\bf c}_u$ in the proof of Theorem \ref{theorem-reduction}.
Combining ${\bf c}_{s/u} := ({\bf c}_s,{\bf c}_u)$ together for the complex eigenvalues outside $\ii \mathbb{R}$, we can write ${\bf c}_{s/u}(0) = {\bf b} + \tilde{\bf c}_{s/u}(0)$, where $\tilde{\bf c}_{s/u}(0)$ are uniquely defined of the order of  $\mathcal{O}(\epsilon^{2N})$ and depend on ${\bf b} \in \cD_s$ in higher orders. By the Implicit Function Theorem, there exists a unique solution ${\bf b} \in \cD_s$ of ${\bf b} = {\bf c}_{s/u}(0) - \tilde{\bf c}_{s/u}(0)$ satisfying the reversibility constraints (\ref{reversibility-constraint}) and this unique ${\bf b}\in \cD_s$ satisfies the bound (\ref{bound-initial}).
\end{itemize}

\begin{remark}
There still exist infinitely many parameters after ${\bf a} \in \cD_c$ have been chosen to satisfy the constraint (\ref{reversibility-constraint}),
namely, ${\rm Re} \; \hat{a}_{m,k}^{(v)}$ and ${\rm Im} \; \hat{a}_{m,k}^{(w)}$ for $(m,k) \in \NN_{\rm odd} \times \ZZ$.
These  parameters of the solution of Theorem \ref{thm1}
must satisfy the bound (\ref{bound-initial})
in Theorem \ref{theorem-reduction}.
\end{remark}

\begin{remark}
Since the local center-stable manifold intersects the plane given by the reversibility constraints (\ref{reversibility-constraint}),
we have thus constructed a family of reversible
solutions on $[-\veps^{-(2N+1)},\veps^{-(2N+1)}]$
while preserving the bound (\ref{property-4}). Tracing the coordinate transformations back to the original variables completes the proof of Theorem \ref{thm1}.
\end{remark}

\section{Proof of Theorem \ref{theorem-time}}
\label{sec-6}

For any given point $(x_0,t_0) \in \mathbb{R} \times (0,\infty)$, we
define a triangular region in the backward light cone of the wave equation:
\begin{equation*}
C_{(x_0,t_0)} := \left\{ (x,t) \in \mathbb{R} \times [0,t_0] : \quad |x-x_0| \leq t_0 - t \right\}.
\end{equation*}
To prove Theorem \ref{theorem-time}, we show that solutions of the semi-linear wave equation (\ref{model}) have a finite propagation speed (bounded in absolute value by 1). We start with the following lemma.

\begin{lemma}
	\label{lemma-wave}
	Let $w \in C^2(\mathbb{R}\times \mathbb{R})$ be a solution of the linear wave equation
	\begin{equation}
	\label{lin-wave}
	\partial_t^2 w(x,t) - \partial_x^2 w(x,t) + \rho(x) w(x,t) = F(x,t),
	\end{equation}
	for a given $F \in C^0(\mathbb{R} \times \mathbb{R})$. For any given point $(x_0,t_0) \in \mathbb{R} \times (0,\infty)$, the energy quantity
	$$
	E(t) := \frac{1}{2} \int_{|x-x_0| \leq t_0 - t} \left[ (\partial_t w(x,t) )^2 + (\partial_x w(x,t) )^2 + \rho(x) w(x,t)^2 \right] dx
	$$
	satisfies
\begin{equation}
\label{energy-inequality}
	E(t_2) \leq E(t_1) + \int_{t_1}^{t_2} \int_{|x-x_0| \leq t_0 - t} F(x,t) \partial_t w(x,t) dx dt,
\end{equation}
	for every $0 \leq t_1 \leq t_2 \leq t_0$.
\end{lemma}

\begin{proof}
	Since the solution $w \in C^2(\mathbb{R}\times \mathbb{R})$ is classical and the integration region is finite, we differentiate $E(t)$ in $t$ and obtain
	\begin{align*}
	E'(t) &= \int_{|x-x_0| \leq t_0 - t} \left[ \partial_t w \partial_t^2 w + \partial_x w \partial_x \partial_t w + \rho w \partial_t w \right] dx \\
	& \quad - \frac{1}{2} \left[ (\partial_t w)^2 + (\partial_x w)^2 + \rho w^2 \right] |_{x = x_0+t_0-t} - \frac{1}{2} \left[ (\partial_t w)^2 + (\partial_x w)^2 + \rho w^2 \right] |_{x = x_0-t_0+t} \\
	&= \int_{|x-x_0| \leq t_0 - t} \left[ \partial_t w \partial_x^2 w + \partial_x w \partial_x \partial_t w + F \partial_t w \right] dx \\
	& \quad - \frac{1}{2} \left[ (\partial_t w)^2 + (\partial_x w)^2 + \rho w^2 \right] |_{x = x_0+t_0-t} - \frac{1}{2} \left[ (\partial_t w)^2 + (\partial_x w)^2 + \rho w^2 \right] |_{x = x_0-t_0+t},
	\end{align*}
where we have used (\ref{lin-wave}). Integration by parts yields
	\begin{align*}
E'(t) &= \int_{|x-x_0| \leq t_0 - t}  F \partial_t w dx
- \frac{1}{2} (\partial_t w - \partial_x w)^2 |_{x = x_0+t_0-t}
- \frac{1}{2} (\partial_t w + \partial_x w)^2 |_{x = x_0-t_0+t} \\
& \qquad \qquad - \frac{1}{2} \rho w^2 |_{x = x_0+t_0-t} - \frac{1}{2} \rho w^2 |_{x = x_0-t_0+t} \\
&\leq \int_{|x-x_0| \leq t_0 - t}  F \partial_t w dx ,
\end{align*}
since $\rho \in \mathcal{X}_0$ is positive in (\ref{functions}). Integration in time on $[t_1,t_2]$ yields (\ref{energy-inequality}).
\end{proof}

Lemma \ref{lemma-wave} enables us to use the energy method in the proof of
uniqueness of solutions of the semi-linear wave equation (\ref{model}) inside
the backward light cone $C_{(x_0,t_0)}$ for any fixed  $(x_0,t_0) \in \mathbb{R} \times (0,\infty)$. The uniqueness is equivalent to the property of the finite propagation speed in  (\ref{model}).

\begin{lemma}
	\label{lemma-uniquness}
	Let $u,v \in C^2(\mathbb{R}\times \mathbb{R})$ be two solutions
	of the semi-linear wave equation (\ref{model}) with
	$$
	u(x,0) = v(x,0), \qquad \partial_t u(x,0) = \partial_t v(x,0), \qquad \forall x \in [x_0-t_0,x_0+t_0].
	$$
	Then, $u(x,t) = v(x,t)$ for every $(x,t) \in C_{(x_0,t_0)}$.
\end{lemma}

\begin{proof}
	Let $w := u - v$. Then $w \in C^2(\mathbb{R}\times \mathbb{R})$ satisfies
	$$
	\partial_t^2 w(x,t) - \partial_x^2 w(x,t) + \rho(x) w(x,t) = \gamma r(x) [u(x,t)^2 + u(x,t) v(x,t) + v(x,t)^2] w(x,t) =: F(x,t),
	$$
	which coincides with (\ref{lin-wave}). By Lemma \ref{lemma-wave} with $t_1 = 0$ and $t_2 = t$, we obtain
	\begin{align*}
	E(t) &\leq E(0) + \int_{0}^{t} \int_{|x-x_0| \leq t_0 - t} F(x,t) \partial_t w(x,t) dx dt \\
	&\leq |\gamma| \| r \|_{L^{\infty}} \int_{0}^{t} \int_{|x-x_0| \leq t_0 - t}  |u^2 + u v + v^2| |w \partial_t w| dx dt \\
	&\leq \frac{3}{2} |\gamma| \| r \|_{L^{\infty}} \left( \|u\|^2_{L^{\infty}(C_{(x_0,t_0)})} + \| v\|^2_{L^{\infty}(C_{(x_0,t_0)})} \right) \int_{0}^{t} \int_{|x-x_0| \leq t_0 - t}  |w \partial_t w| dx dt,
	\end{align*}
since $E(0) = 0$ and $C_{(x_0,t_0)}$ is compact. For $\rho \in \mathcal{X}_0$ given by (\ref{functions}), the energy quantity $E(t)$ is coercive and we have
$$
\| \partial_t w(\cdot,t) \|^2_{L^2(B_{(x_0,t_0,t)})} + \| \partial_x w(\cdot,t) \|^2_{L^2(B_{(x_0,t_0,t)})} + \rho_0 \| w(\cdot,t) \|^2_{L^2(B_{(x_0,t_0,t)})} \leq 2 E(t),
$$
where $B_{(x_0,t_0,t)} := \{ x \in \mathbb{R} : \; |x-x_0| \leq t_0 - t\}$. Hence, by Cauchy--Schwarz inequality, we obtain
	\begin{align*}
E(t) &\leq C \left( \|u\|^2_{L^{\infty}(C_{(x_0,t_0)})} + \| v\|^2_{L^{\infty}(C_{(x_0,t_0)})} \right) \int_{0}^{t} E(t) dt,
\end{align*}
for some $C > 0$ that depends on $\rho,r \in \mathcal{X}_0$. By Gronwall's inequality, this yields $E(t) = 0$ for every $t \in [0,t_0]$ so that
$w(\cdot,t) = 0$ in $H^1(B_{(x_0,t_0,t)})$ and $\partial_t w(\cdot,0) = 0$ in $L^2(B_{(x_0,t_0,t)})$ for every $t \in [0,t_0]$. By Sobolev's embedding of $H^1(\mathbb{R})$ into $C^0_b(\mathbb{R})$, this implies that $w(x,t) = 0$ for every $x \in B_{(x_0,t_0,t)}$ and $t \in [0,t_0]$, that is, for every $(x,t) \in C_{(x_0,t_0)}$.
\end{proof}

\begin{remark}
	Well-posedness of the initial-value problem for the semi-linear wave equation (\ref{model}) with $\rho, r \in \mathcal{X}_0$ and bootstrapping can be used to state that the first partial derivatives of $u(x,t)$ and $v(x,t)$ in $(x,t)$ are also equal for every $(x,t) \in C_{(x_0,t_0)}$.
\end{remark}

\begin{remark}
	Theorem \ref{theorem-time} is a restatement of Lemma \ref{lemma-uniquness}, where $$v(x,t) = v_{\rm ext}(x-c_gt,l_0 x - \omega t,x)$$ is constructed based on Theorem \ref{thm1} and an arbitrary function $\phi$. The function $v_{\rm ext}(\xi,z,x)$ belongs to $C^2(\mathbb{R},\mathcal{X})$ which is weaker than $C^2(\mathbb{R}\times \mathbb{R})$ for the function $v(x,t)$ in variables $(x,t)$. However, Lemmas \ref{lemma-wave} and \ref{lemma-uniquness} can be extended by a simple density argument to functions for which 
	the energy $E(t)$ is bounded in $C_{(x_0,t_0)}$ and for which the linear wave equation (\ref{lin-wave}) is satisfied almost everywhere $C_{(x_0,t_0)}$. 
	Functions $u(x,t)$ and $v(x,t)$ constructed from $v_{\rm ext}(\xi,z,t)$ as in Theorem \ref{theorem-time} belong to the required function space. 
\end{remark}

\bibliographystyle{alpha}
\bibliography{movbib}

\end{document}